\documentclass{amsart}

\usepackage{amsmath,amssymb,latexsym}
\usepackage{amsfonts,amstext,amsthm,amscd}
\usepackage{graphicx,comment}
\usepackage{enumerate}
\usepackage[all]{xy}
\usepackage{ifpdf}

\usepackage{float}
\usepackage{listings}
\usepackage{subcaption}
\usepackage{tikz}
\usepackage{pgfplots}
\pgfplotsset{compat=1.18}

\usepackage{hyperref}

\theoremstyle{plain}
\newtheorem{theorem}{Theorem}[section]

\newtheorem{lemma}[theorem]{Lemma}
\newtheorem*{lemma*}{Lemma}

\theoremstyle{definition}

\newtheorem{example}[theorem]{Example}

\theoremstyle{remark}
\newtheorem{remark}[theorem]{Remark}

\theoremstyle{plain}
\newtheorem{conjecture}[theorem]{Conjecture}

\newcommand{\abs}[1]{\left\vert#1\right\vert}
\DeclareMathOperator{\arccot}{arccot}
\DeclareMathOperator{\Hess}{Hess}

\newcommand{\C}{\mathbb C}

\newcommand{\imag}{\mathrm{Im}}
\newcommand{\R}{\mathbb R}
\newcommand{\tr}{\triangle}

\newcommand{\To}{\longrightarrow}

\title[On Gaussian Curvature of Graphs of Cubic Polynomials]{On Gaussian Curvature of Graphs \\ of Cubic Polynomials}

\author[B.S. Cabrera]{Baruc S. Cabrera}
\address{Divisi\'on de Ciencias Naturales y Exactas \\
Universidad de Guanajuato \\
Jalisco s/n, Mineral de Valenciana \\
Guanajuato, Gto. 36240, M\'exico}
\email{b.cabrera@ugto.mx}

\author[M. Cruz-L\'opez]{Manuel Cruz-L\'opez}
\address{Departamento de Matem\'aticas \\
Universidad de Guanajuato \\
Jalisco s/n, Mineral de Valenciana \\
Guanajuato, Gto. 36240, M\'exico}
\email{manuelcl@ugto.mx}

\author[O. Romero-Germ\'an]{Otto Romero-Germ\'an}
\address{CIMAT \\
Jalisco s/n, Mineral de Valenciana \\
Guanajuato, Gto. 36240, M\'exico}
\email{otto.romero@cimat.mx}


\date{\today}

\begin{document}

\begin{abstract}
The Gaussian curvature of the graphs of cubic polynomials $f(z)=(z-a)(z-b)(z-c)$, with triangle of zeros $\tr=\tr(a,b,c)$, is analyzed. We introduce the \emph{composed Pythagorean mean} $\mathtt{B}_{\tr}$, a combination of the three classical Pythagorean means, which plays a central role in the description of the critical points of the Gaussian curvature. When $\tr$ is not equilateral, the curvature has exactly five critical points: a global maximum at the barycenter, two local minima, and two saddle points, forming a rhombus whose vertices lie on the axes of the Siebeck-Marden ellipse of the triangle. We also initiate a comparative study of $\mathtt{B}_{\tr}$ with the skew of the triangle by describing a family of polynomials whose zeros determine a specific family of isosceles triangles.

The cubic case provides a concrete model strongly suggesting that the critical geometry of the Gaussian curvature of the graph can reveal both analytic and algebraic properties of the underlying holomorphic function.
\end{abstract}

\maketitle


\section[Introduction]{Introduction}
\label{introduction}

The study of curves and surfaces has been a central topic in real and complex differential geometry since the works of C.F. Gauss and B. Riemann. The Gaussian curvature plays a central role in this development, both for real surfaces and for holomorphic curves. The graph of a holomorphic function defined in the complex plane can be seen as a holomorphic curve. Its Gaussian curvature is a classical object of study, as presented in \cite{Gri}. The first nontrivial example exhibiting a rich geometrical behavior is the case of a cubic polynomial, whose basic geometry (see e.g., \cite{Mar}) is related to beautiful geometric objects such as the Siebeck-Marden ellipse, as well as an important triangle invariant: the skew. This work explores how classical differential-geometric invariants of the graph of a complex polynomial $f$, such as the Gaussian curvature, can reveal subtle algebraic and, geometric features encoded by the zeros of $f$.

Given a cubic polynomial $f=u+iv$ defined in the complex plane $\C$, its graph, $\Gamma(f)=\{(z,f(z))\in \C^2:z\in \C\}$, is a holomorphic curve. The coordinate functions of $f$ provide a natural parametrization of $\Gamma(f)$ as a regular surface immersed in $\R^4$, with the induced metric from $\R^4$. The formula for the Gaussian curvature of $\Gamma(f)$ at any point of the graph,
\[ K = -\frac{\Delta E}{2E^3} \]
was obtained in \cite{CR}. Here, $E$ is the coefficient of the first fundamental form and $\Delta E$ is the Laplacian of $E$.

In this article, we examine the Gaussian curvature of the graph of a cubic polynomial $f(z)=(z-a)(z-b)(z-c)$. Such a fundamental geometric invariant is completely described by the zeros of $f$, which can be located at any three noncollinear points $a,b,c$ in the complex plane.

If we write $A=a+b+c$, $B=ab+ac+bc$ and $C=abc$, $f(z)$ can be expressed as
\[ f(z) = z^3-Az^2+Bz-C. \]

Denote by $T=\tr(a,b,c)$ the triangle determined by the vertices at the points
$a,b,c\in \C$. Associated with $T$ define:
\begin{itemize}
\item $\mathtt{b}_\tr:=A/3$, the barycenter of $T$.
\item $\mathtt{B}_{\tr}:=\mathtt{b}_{\tr}^2 - B/3$.
\end{itemize}

We call the quantity $\mathtt{B}_{\tr}$ the \textsf{composed Pythagorean mean}, since it can be described as a nice combination of the three classical Pythagorean means: the arithmetic mean, the geometric mean, and the harmonic mean (see Section~\ref{skew_composed-pythagorean-mean} for details). This composed Pythagorean mean
\[ \mathtt{B}_{\tr} = \frac{1}{9} \big( a^2 + b^2 + c^2 - ab - ac - bc \big) \]
measures how far the triangle is from being equilateral. That is,

\begin{remark}
\label{T_equilateral}
\[ T \text{ is equilateral } \; \iff \; a^2+b^2+c^2 = ab+ac+bc \iff \; \mathtt{B}_\tr=0. \]
\end{remark}

The Gaussian curvature of the graph of $f$ depends explicitly on the zeros of $f$ and is given by
\[
K(z) = - C_0\frac{\abs{z-\mathtt{b}_\tr}^2}{\big[ (1/9)+\abs{(z-\mathtt{b}_\tr)^2-\mathtt{B}_\tr}^2 \big]^3}
\]
where $C_0:=2^3/3^4$.

An analysis of critical points calculations and classification is subsumed in the next (see Theorem~\ref{critical-points_theorem} and Theorem~\ref{type_critical-points}):
\smallskip

\noindent \textsf{First Main Theorem:}
If $f(z)=(z-a)(z-b)(z-c)$ is a cubic polynomial with nondegenerate and not equilateral triangle of zeros $T=\tr(a,b,c)$, then, the Gaussian curvature of the graph of $f$ has five critical points, corresponding to 1 maximum, 2 minima, and 2 saddle points. Furthermore, the 5 critical points depend on $\mathtt{b}_{\tr}$ and $\mathtt{B}_{\tr}$.

According to the Siebeck-Marden Theorem (see \cite{Mar}), the roots
$z_1=\mathtt{b}_{\tr} + \sqrt{\mathtt{B}_{\tr}}$ and
$z_2=\mathtt{b}_{\tr} - \sqrt{\mathtt{B}_{\tr}}$ of $f'(z)$ are the foci of the inscribed ellipse in $T$  tangent to the midpoints of the sides of $T$. The zero of the second derivative located at $z=\mathtt{b}_{\tr}$ is the midpoint of the segment $[z_1,z_2]$, i.e., it is the center of the ellipse and the barycenter of the triangle.

When $T$ is equilateral, the Gaussian curvature of the graph of $f$ has a circumference of critical points with center at $\mathtt{b}_\tr$ and radius $1/\sqrt[4]{45}$. This circumference is concentric with the Siebeck-Marden ellipse (a circumference, in this case).

When $T$ is not equilateral, i.e., $\mathtt{B}_\tr\neq 0$, we prove the following:

\smallskip
\noindent \textsf{Theorem \ref{coincidence_diagonals}}:
If $f(z)=(z-a)(z-b)(z-c)$ is any cubic polynomial with $\mathtt{B}_\tr\neq 0$, then, the five critical points of the Gaussian curvature of the graph of $f$ are located along both axes of the Siebeck-Marden ellipse. In particular, the barycenter corresponds to the center of the ellipse.

\smallskip
The Figure \ref{fig:scalene_intro}(A) shows the graph of the Gaussian curvature with the triangle of zeros in blue and the critical quadrilateral in green. Figure~\ref{fig:scalene_intro}(B) shows the triangle of zeros in blue with $\mathtt{B}_\tr\neq 0$. The critical points of the Gaussian curvature are in green. The Siebeck-Marden ellipse is shown in red, and its foci are shown in black. Two of the critical points are located on the focal axis. The distance from each focus to the nearest critical point is approximately $0.032$, which is why they appear to coincide in the figure.
\begin{figure}[!htbp]
\centering
\begin{subfigure}[b]{0.45\textwidth}
\centering
\includegraphics[scale=0.175]{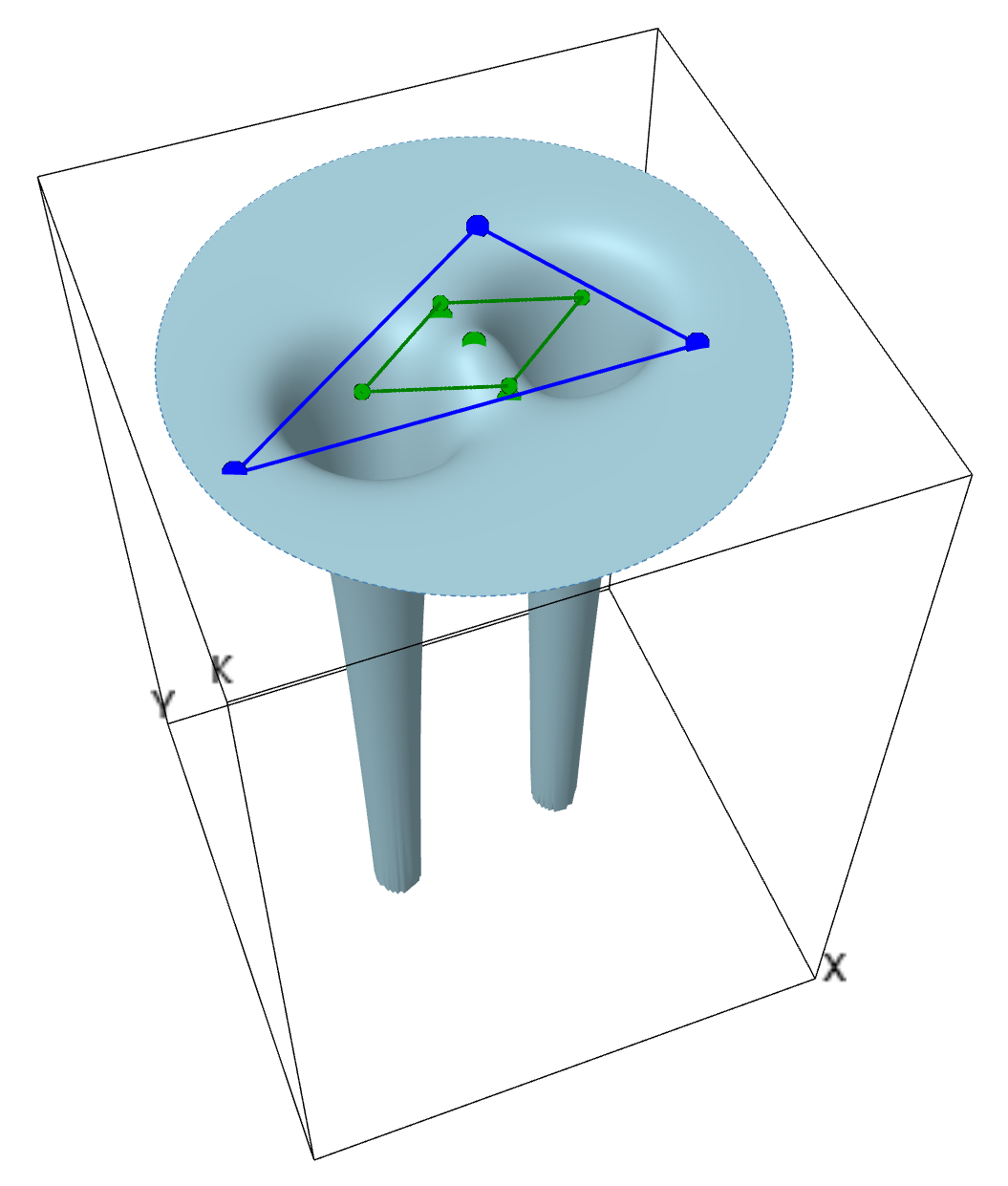}
\caption{}
\end{subfigure}
\begin{subfigure}[b]{0.45\textwidth}
\centering
\raisebox{1.5cm}{%
\resizebox{\linewidth}{!}{%
\begin{tikzpicture}
\begin{axis}[
    axis equal image,
    xmin=-1.45, xmax=1.45,
    ymin=-0.25, ymax=1.65,
    axis lines=center,
    xlabel={},
    ylabel={},
    ticks=none,
]
\addplot[thick,blue] coordinates {
    (-1.25,0)
    (1.25,0)
    (0.5,1.25)
    (-1.25,0)
};
\addplot[thick,solid,color=red,domain=0:360,samples=300,]
  ({0.16667 + 0.74904*cos(x)*cos(10.16157) - 0.40145*sin(x)*sin(10.16157) },
  {0.41667 + 0.74904*cos(x)*sin(10.16157) + 0.40145*sin(x)*cos(10.16157) });
\begin{scope}[
    shift={(axis cs:0,0)},
    x={(axis direction cs:1,0)},
    y={(axis direction cs:0,1)}]
\draw[blue,fill=blue] (-1.25,0) circle (.5ex);
\draw[blue,fill=blue] (1.25,0) circle (.5ex);
\draw[blue,fill=blue] (0.5,1.25) circle (.5ex);
\draw (-1.35,-0.15) node[right,black] {$a$};
\draw (1.15,-0.15) node[right,black] {$b$};
\draw (0.55,1.25) node[right,black] {$c$};
\draw[thick,color=orange] (-0.75,0.2522) -- (1.15,0.5929);
\draw[thick,color=orange] (0.26822,-0.15) -- (0.0352,1.15);
\draw[thick,color=green,densely dotted] (0.8213,0.5340) -- (0.2284,0.07213);
\draw[thick,color=green,densely dotted] (0.8213,0.5340) -- (0.1049,0.7612);
\draw[thick,color=green,densely dotted] (-0.4880,0.2993) -- (0.2284,0.07213);
\draw[thick,color=green,densely dotted] (-0.4880,0.2993) -- (0.1049,0.7612);
\draw[green,fill=green] (0.8213,0.5340) circle (.35ex);
\draw[green,fill=green] (-0.4880,0.2993) circle (.35ex);
\draw[green,fill=green] (0.2284,0.07213) circle (.35ex);
\draw[green,fill=green] (0.1049,0.7612) circle (.35ex);
\draw[black,fill=black] (0.7892,0.5283) circle (.4ex);
\draw[black,fill=black] (-0.4559,0.301) circle (.4ex);
\draw[red,fill=red] (0,0) circle (.35ex);
\draw[red,fill=red] (-0.375,0.625) circle (.35ex);
\draw[red,fill=red] (0.875,0.625) circle (.35ex);
\draw[black,fill=black] (0.1666,0.4166) circle (.4ex);
\end{scope}
\end{axis}
\end{tikzpicture}%
}%
}%
\caption{}
\end{subfigure}
\caption{Curvature and geometric configuration for a scalene triangle: (A) graph of the Gaussian curvature; (B) critical points and Siebeck--Marden ellipse for $\triangle(a,b,c)=\triangle\bigl(-\tfrac54,\tfrac54,\tfrac12+\tfrac{5i}{4}\bigr)$}
\label{fig:scalene_intro} 
\end{figure}
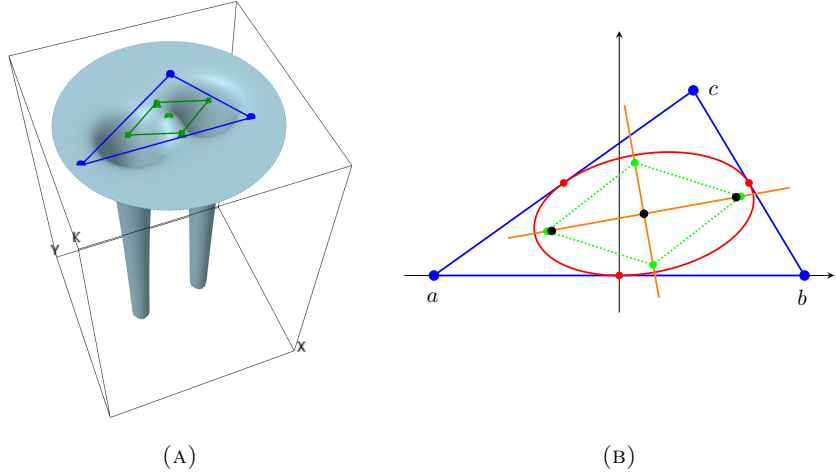

In the last part of this work, we carefully examine the quantity $\mathtt{B}_\tr$, and compare it with the \textsf{skew} of the triangle $T$, which measures how far the triangle is from being equilateral. In the equilateral case, $\mathrm{skew}(T)=1$ and $\mathtt{B}_\tr=0$.

The skew was introduced by J.H. Hubbard (see \cite{Hub}) for studying quasiconformality in the complex plane (see also \cite{AHH}). We study the Gaussian curvature of a family of polynomials of the form $f(z)=(z-it)(z+it)(z-\lambda)$ in Section~\ref{skew_composed-pythagorean-mean}, where $\lambda,t\in \R$ and $t>0$. In this case, the associated triangles $T_{\lambda,t}$ are isosceles and $\mathtt{B}_\tr$ is a real number. For this class of triangles we prove that
\[ \mathrm{skew}(T_{\lambda,t})^2 = 1 + C(\ell) \abs{\mathtt{B}_\tr}, \]
where $\ell$ is the smallest side of the triangle and $C(\ell)=9/\ell^2$.

In Figure \ref{skew_graph}, we plot the graphs of $\mathrm{skew}(T_\tau)^2$ (blue) and $1+\frac{9\abs{\mathtt{B}_\tr}}{\ell^2}$ (orange) for the family of triangles $T_\tau=\triangle(-1,1,\tau)$, with $\tau\in\C$ and $\imag(\tau)>0$. The curves in the $XY$-plane indicate the loci where two sides of $T_\tau$ have equal length, while the black point (in both surfaces) corresponds to an equilateral triangle. These images motivate the search for a more precise relationship between $\mathrm{skew}(T_\tau)$ and $\abs{\mathtt{B}_\tr}$, analogous to the one obtained for the family of isosceles triangles $T_{\lambda,t}$.
\begin{figure}[ht]
	\centering
	\includegraphics[scale=0.15]{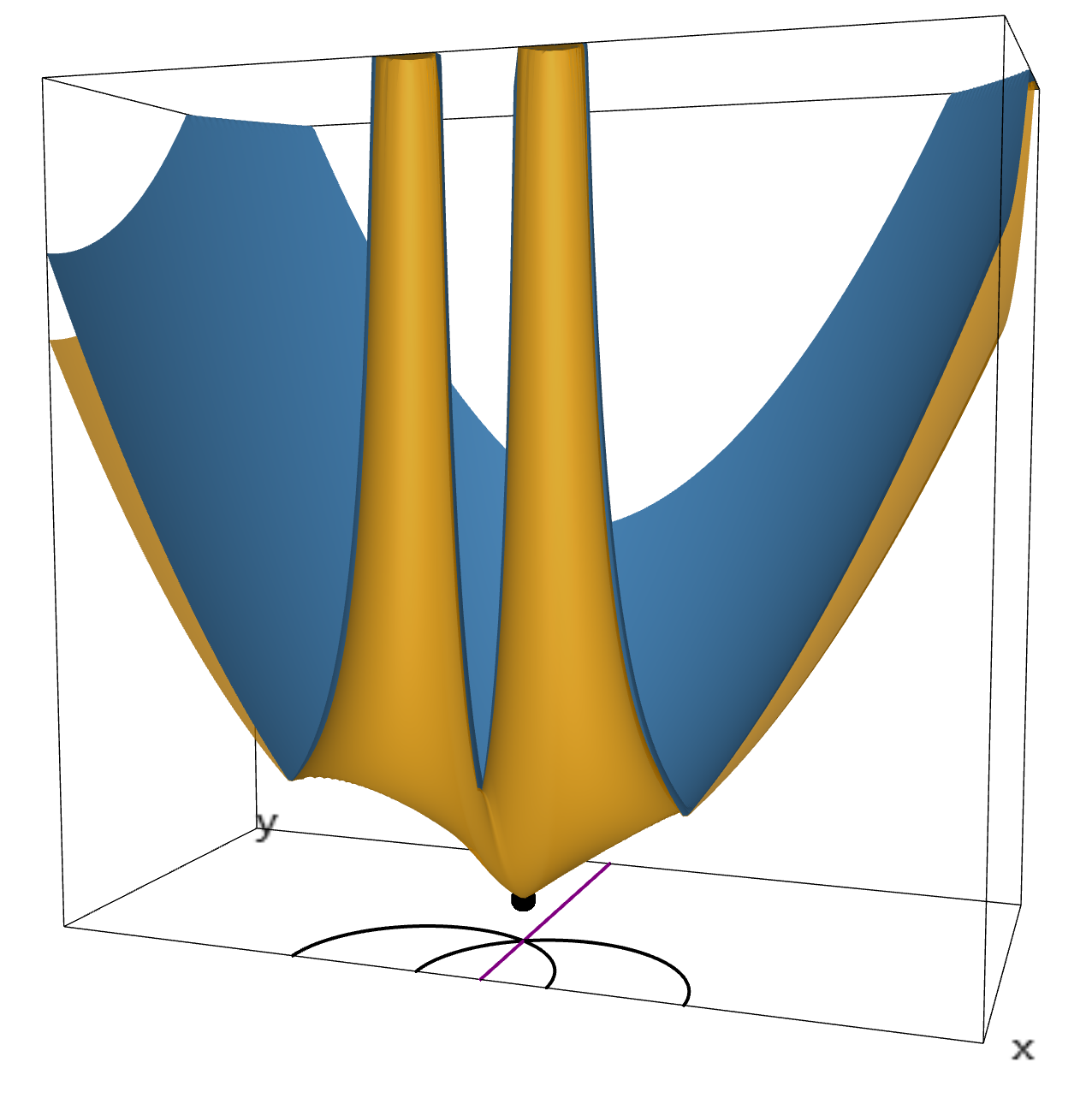}
    \caption{Graphs of $\operatorname{skew}(T_\tau)^2$ and $1+9\abs{\mathtt{B}_{\tr}}/\ell^2$ for $\operatorname{Im}(\tau)>0$}
	\label{skew_graph}
\end{figure}

\noindent \textbf{The Skew Conjecture \ref{skew_conjecture}:} Given a cubic polynomial $f(z)$ with nondegenerate triangle of zeros $T=\tr(a,b,c)$ in the complex plane, the skew of $T$ can be approximated by $\abs{\mathtt{B}_\tr}$, as
\[ \mathrm{skew}(T)^2 \approx 1 + C(\ell) \abs{\mathtt{B}_\tr}, \]
where $\ell$ is the smallest side of $T$ and the constant $C(\ell)$ depends only on $\ell$.

\smallskip
The study carried out in this work suggests further development of the theory in several aspects. On the one hand, the analysis of the Gaussian curvature can be extended to higher-degree polynomials, other types of holomorphic functions, or even meromorphic functions. On the other hand, deeper relations between the composed Pythagorean mean and the skew can be pursued. These considerations will be part of future investigations.

Section \ref{gaussian-curvature_cubic-polynomials} presents the basic formula for calculating the Gaussian curvature of a cubic polynomial, as well as some of its invariance properties of the curvature. Section~\ref{critical-points_gaussian-curvature} focuses on calculations of the critical points of the Gaussian curvature, and analyzes their nature. In Section~\ref{critical_quadrilaterals}, we study the different possibilities that appear in the plane when describing the quadrilaterals determined by the critical points of the curvature, as well as the Siebeck-Marden ellipse inscribed in the triangle determined by the zeros of the cubic polynomial. Finally, in Section~\ref{skew_composed-pythagorean-mean} we compare the composed Pythagorean mean  with the skew of the triangle of zeros of the polynomial.

\section[Gaussian Curvature of Graphs of Cubic Polynomials]{Gaussian Curvature of Graphs of Cubic Polynomials}
\label{gaussian-curvature_cubic-polynomials}

This section presents the basic formula to calculate the Gaussian curvature at points of the graph of a cubic polynomial, as well as some of its invariant properties.

\subsection[General Form of the Gaussian Curvature]{General Form of the Gaussian Curvature}
\label{general_gaussian-curvature}

Let $f:\Omega\To \C$ be a holomorphic function defined on a domain
$\Omega\subset \C$. We write
\[ f(x,y) = u(x,y) + i v(x,y),\quad (z=x+iy\in \Omega) \]
where $u,v:\Omega\To \R$ are infinitely differentiable functions. The \textsf{graph} of $f$ is by definition the set
\[ \Gamma(f):= \left\{ (z,f(z))\in \C^2 : z\in \Omega \right\}. \]

This description enables us to provide complex and real parametrizations of $\Gamma(f)$. That is, as a holomorphic curve, $\Gamma(f)$ admits the parametrization
\[ \Psi(z) = (z,f(z)) \qquad (z\in \Omega) \]
and, as a subset of $\R^4$, $\Gamma(f)$ is a real surface which admits the parametrization $\Psi:\Omega\To \R^4$ given by
\[ \Psi(x,y) = (x,y,u(x,y),v(x,y)). \]

Since $\Gamma(f)$ is a differentiable surface, we can endow it with the induced Riemannian metric from $\R^4$. The Gaussian curvature $K$ of the graph of $f$ at the point $(z,f(z))\in \Gamma(f)$ is given by
\[ K = -\frac{\Delta E}{2E^3}, \]
where $E(x,y)=1+u_x^2+v_x^2$.

\subsection[Cubic Polynomials]{Cubic Polynomials}
\label{cubic_gaussian-curvature}

Suppose that $a,b,c$ are any three noncollinear points in the complex plane and denote by $T=\tr(a,b,c)$ the triangle determined by the vertices at
$a,b,c\in \C$. Define the cubic polynomial
 \[ f(z) = (z-a)(z-b)(z-c). \]
If $A=a+b+c$, $B=ab+ac+bc$ and $C=abc$, then $f$ can be written as
\[ f(z) = z^3 - Az^2 + Bz - C, \]
with derivative given by
\[ f'(z) = 3z^2-2Az+B = 3[(z-\mathtt{b}_{\tr})^2-\mathtt{B}_{\tr}], \]
where $\mathtt{b}_{\tr}=A/3$ is the barycenter of $T$, and $\mathtt{B}_{\tr}=\mathtt{b}_{\tr}^2 - B/3$. Also,
\begin{equation}
\label{focal_points}
f'(z)=0 \quad \text{if, and only if,}\quad z=\mathtt{b}_{\tr}\pm \sqrt{\mathtt{B}_{\tr}}.
\end{equation}

The second derivative of $f$ is
\[ f''(z)=6z-2A=6(z-\mathtt{b}_{\tr}). \]

The Gaussian curvature at any point of the graph of $f$ is given by:
\[
K(z) = - \frac{2^3\cdot \abs{z-\mathtt{b}_{\tr}}^2}{3^4\big[ (1/9) + \abs{(z-\mathtt{b}_{\tr})^2-\mathtt{B}_{\tr}}^2 \big]^3}.
\]

When $\mathtt{b}_\tr=0$, we obtain the simpler formula
\[
K(z) = - C_0\frac{\abs{z}^2}{\bigg[ (1/9) + \abs{z^2 - \mathtt{B}_{\tr}}^2 \bigg]^3},
\]
where $C_0=2^3/3^4$.

\subsection[Invariance Properties of the Gaussian Curvature]{Invariance Properties of the Gaussian Curvature}
\label{invariance_properties_gaussian-curvature}

Consider a translation by $t$, $t\in \C$, of the triangle $T$ as $\tr_{+t}:=\tr(a+t,b+t,c+t)$. Also, for $\theta\in [0,2\pi]$, consider a rotation by $\theta$,
$\tr_\theta:=\tr\big(e^{i\theta}a,e^{i\theta}b,e^{i\theta}c\big)$, of $T$.
\begin{lemma}
\label{B_invariant}
The following relations hold:
\begin{align*}
\mathtt{b}_{\tr_{+t}} &= \mathtt{b}_{\tr} + t, \\
\mathtt{b}_{\tr_\theta} &= e^{i\theta} \mathtt{b}_{\tr},
\end{align*}
and
\begin{align*}
\mathtt{B}_{\tr_{+t}} &= \mathtt{B}_{\tr}, \\
\mathtt{B}_{\tr_\theta} &= e^{2i\theta} \mathtt{B}_{\tr}.
\end{align*}
\end{lemma}


The Gaussian curvature of the graphs of cubic polynomials determined by the triangles $\tr_{+t}$ and $\tr_\theta$ will be denoted by $K_{+t}$ and $K_\theta$, respectively.

\begin{lemma}
\label{symmetry_curvature}
The following relations hold:
\begin{enumerate}[(a)]
\item $K_{+t}(z) = K(z-t), \; z,t\in \C$.
\item $K_\theta(z) = K(e^{-i\theta}z), \; z\in \C, \theta\in [0,2\pi]$.
\item Write $\mathtt{B}_\tr=B_1+iB_2$ with $B_2\neq 0$. Then, there exists
$\theta\in (0,2\pi)\setminus \{\frac{\pi}{2},\frac{3\pi}{2}\}$ such that the equality  $\mathtt{B}_{\tr_\theta}=B_1^\theta$ hold for triangle $\tr_\theta$. That is, $B_2^\theta=0$ and $B_1^\theta \neq 0$ for triangle $\triangle_\theta$.
\end{enumerate}
\end{lemma}

\begin{proof}
\begin{enumerate}
\item[(a)] By Lemma \ref{B_invariant} we obtain:
\begin{align*}
K_{+t}(z) &= - C_0 \frac{\abs{z-\mathtt{b}_{\tr_{+t}}}^2}{\big[ (1/9) +
\abs{(z-\mathtt{b}_{\tr_{+t}})^2-\mathtt{B}_{\tr_{+t}}}^2 \big]^3} \\
&= - C_0 \frac{\abs{z-\mathtt{b}_{\tr}-t}^2}{\big[ (1/9) +
\abs{(z-\mathtt{b}_{\tr}-t)^2-\mathtt{B}_{\tr}}^2 \big]^3}.
\end{align*}
Hence,
\begin{align*}
K_{+t}(z+t) &= - C_0 \frac{\abs{z-\mathtt{b}_{\tr}}^2}{\big[ (1/9) + \abs{(z-\mathtt{b}_{\tr})^2-\mathtt{B}_{\tr}}^2 \big]^3} \\
&= K(z).
\end{align*}

\item[(b)] Again, by Lemma \ref{B_invariant} we obtain:
\begin{align*}
K_\theta(z) &= - C_0 \frac{\abs{z-\mathtt{b}_{\tr_\theta}}^2}{\big[ (1/9) + \abs{(z-\mathtt{b}_{\tr_\theta})^2-\mathtt{B}_{\tr_\theta}}^2 \big]^3} \\
&= - C_0 \frac{\abs{z-e^{i\theta}\mathtt{b}_{\tr}}^2}{\big[ (1/9) + \abs{(z-e^{i\theta}\mathtt{b}_{\tr})^2-e^{2i\theta}\mathtt{B}_{\tr}}^2 \big]^3} \\
&= - C_0 \frac{\abs{e^{-i\theta}z-\mathtt{b}_{\tr}}^2}{\big[ (1/9) + \abs{(e^{-i\theta}z-\mathtt{b}_{\tr})^2-\mathtt{B}_{\tr}}^2 \big]^3} \\
&= K(e^{-i\theta}z).
\end{align*}

\item[(c)] By Lemma \ref{B_invariant}, $\mathtt{B}_{\tr_\theta}=e^{2i\theta} \mathtt{B}_{\tr}$. Hence,
\[ \sin(2\theta)\cdot B_1 + \cos(2\theta)\cdot B_2 = B_2^\theta, \]
so we can choose
\[ \theta=\frac{1}{2} \arccot \bigg(-\frac{B_1}{B_2}\bigg). \]
\end{enumerate}
\end{proof}

\section[Critical Points of the Gaussian Curvature]{Critical Points of the Gaussian Curvature}
\label{critical-points_gaussian-curvature}

In this section, we compute the critical points of the Gaussian curvature of the graph of a cubic polynomial and determine whether they are maxima, minima, or saddle points.

Suppose that $w:=x+iy$, with $x,y \in \R$. Setting $w=z-\mathtt{b}_{\tr}$ and writing $\mathtt{B}_{\tr}=B_1+iB_2$, we can express $K(x,y)$ in the form
\begin{align*}
K(x,y)
&= -C_0\frac{\abs{w}^2}{\bigg[ (1/9)+\abs{(x^2-y^2-B_1)+i(2xy-B_2)}^2 \bigg]^3} \\
&= - C_0\frac{\abs{w}^2}{\bigg[ (1/9) + (x^2-y^2-B_1)^2 + (2xy-B_2)^2 \bigg]^3} \\
&= - C_0\abs{w}^2\cdot B(x,y),
\end{align*}
where $B(x,y): = \frac{1}{A(x,y)^3}$ and
\[ A(x,y): = (1/9) + (x^2-y^2-B_1)^2 + (2xy-B_2)^2. \]

The first partial derivative with respect to $x$ is given by:
\begin{equation}
\label{Kx}
K_x(x,y) = \frac{-C_0}{A(x,y)^4} \bigg[ 2xA(x,y) - 3(x^2+y^2)A_x(x,y) \bigg].
\end{equation}

So,
\[ K_x(x,y) = 0\quad \Longleftrightarrow \quad 2xA(x,y)-3(x^2+y^2)A_x(x,y) = 0. \]

Similarly, the first partial derivative with respect to $y$ is:
\begin{equation}
\label{Ky}
K_y(x,y) = \frac{-C_0}{A(x,y)^4} \bigg[ 2yA(x,y) - 3(x^2+y^2)A_y(x,y) \bigg].
\end{equation}

So,
\[ K_y(x,y) = 0\quad \Longleftrightarrow \quad 2yA(x,y) - 3(x^2+y^2)A_y(x,y) = 0. \]

In order to obtain the critical points, we must examine the system:
\begin{align}
2xA(x,y) - 3(x^2+y^2)A_x(x,y) &= 0, \label{critical-points_loci1}\\
2yA(x,y) - 3(x^2+y^2)A_y(x,y) &= 0. \label{critical-points_loci2}
\end{align}

\subsection[Geometric Analysis of Critical Points: Symmetries of the Gaussian Curvature]{Geometric Analysis of Critical Points: Symmetries of the Gaussian Curvature}
\label{critical-points_analysis}

For any $(x,y)\in \R^2$, $A(x,y)$ is symmetric with respect to the origin and $A_x,A_y$ are antisymmetric, i.e.,
\[ A_x(-x,-y)=-A_x(x,y)\quad \text{and}\quad A_y(-x,-y)=-A_y(x,y). \]
That means that, if $(x,y)\in \R^2$ is a solution of the system of equations \eqref{critical-points_loci1} -- \eqref{critical-points_loci2}, then
$(-x,-y)$ is also a solution of the system.

It is clear that $(0,0)$ is a solution of the system of equations. If $(x,y)\neq (0,0)$, then
\begin{equation}
A_x(x,y) = \frac{2xA(x,y)}{3(x^2+y^2)}, \qquad
A_y(x,y) = \frac{2yA(x,y)}{3(x^2+y^2)}.
\label{Ax_Ay}
\end{equation}

Since $A(x,y)\neq 0$ for any $(x,y)\in \R^2$, the equation is equivalent to
\[ \frac{A_x(x,y)}{x} = \frac{A_y(x,y)}{y}, \]
or,
\[ yA_x(x,y) - xA_y(x,y) = 0. \]

Recall that
\[ A(x,y) = (1/9) + ([x^2-y^2]-B_1)^2 + (2xy-B_2)^2, \]
that is,
\[ A(x,y) = (1/9) + (x^2 + y^2)^2 - 2[B_1x^2 + 2B_2xy - B_1y^2] + (B_1^2 + B_2^2). \]

So,
\[ yA_x(x,y) - xA_y(x,y) = 4[ B_2x^2 - 2B_1xy - B_2y^2 ]. \]

Hence, the critical points are on the degenerate conic (two lines)
\begin{equation}
\label{degenerated_conic}
B_2x^2 - 2B_1xy - B_2y^2 = 0.
\end{equation}

When $B_1\neq 0$ and $B_2=0$, this equation simplifies to $xy=0$, i.e., the coordinate axes.

\subsection[Computation of Critical Points]{Computation of Critical Points}
\label{computation_critical-points}

\subsubsection[Case: $B_1=B_2=0$]{Case: $B_1=B_2=0$}
\label{equilateral_triangles}

When $B_1=B_2=0$,
\[ A(x,y)=1/9+(x^2+y^2)^2 \]
and
\[ A_x(x,y)= 4x(x^2+y^2), \quad A_y(x,y)= 4y(x^2+y^2). \]

The system \eqref{critical-points_loci1} -- \eqref{critical-points_loci2} reduces to
\begin{align*}
2xA(x,y) - 12x(x^2+y^2)^2 &= 0, \\
2yA(x,y) - 12y(x^2+y^2)^2 &= 0,
\end{align*}
and has a trivial solution at $x=y=0$.

When $x\neq 0$ and $y\neq 0$, the previous system is equivalent to a single equation
\[ A(x,y) - 6(x^2+y^2)^2 = 0, \]
whose solution set is the circumference with center at $(0,0)$ and radius $1/\sqrt[4]{45}$, i.e., the locus of points $(x,y)\in \R^2$ such that
\[ \sqrt{x^2+y^2} = \frac{1}{\sqrt[4]{45}}. \]

\subsubsection[Case $B_1\neq0$, $B_2=0$]{Case $B_1\neq0$, $B_2=0$}
\label{isosceles_triangles}

In this case,
\[  A(x,y) = 1/9 + (x^2 + y^2)^2 - 2B_1x^2 + 2B_1y^2 + B_1^2,  \]
and
\[  A_x(x,y) = 4x(x^2 + y^2) - 4B_1x,  \]
\[  A_y(x,y) = 4y(x^2 + y^2) + 4B_1y.  \]

Hence,
\begin{equation}
 A(x,0) = 1/9 + x^4 - 2B_1x^2 + B_1^2, \quad A(0,y) = 1/9 + y^4 + 2B_1y^2 + B_1^2.
\label{B2Z_1}
\end{equation}
\begin{equation}
 A_x(x,0) = 4x^3 - 4B_1x, \quad  A_y(x,0) = 0.
\label{B2Z_2}
\end{equation}
\begin{equation}
A_x(0,y) = 0, \quad  A_y(0,y) = 4y^3 + 4B_1y.
\label{B2Z_3}
\end{equation}

The system \eqref{critical-points_loci1} -- \eqref{critical-points_loci2} evaluated at points $(x,0)$ with $x\neq 0$, simplifies to:
\[ 2xA(x,0) - 3x^2A_x(x,0) = 0.  \]

Substituting \eqref{B2Z_1} and \eqref{B2Z_2} in the above equation we obtain:
\[ - 45x^4 + 36B_1x^2 + (1 + 9B_1^2) = 0. \]

This equation has real solutions:
\[ X_1 = \frac{1}{\sqrt{15}} \sqrt{\sqrt{81 B_1^2 + 5} + 6B_1},
\
X_2 = -\frac{1}{\sqrt{15}} \sqrt{\sqrt{81 B_1^2 + 5} + 6B_1}.
\]
Note that
\[ \sqrt{81 B_1^2 + 5} + 6B_1 \geq 0, \; B_1\in \R, \]
so, the system \eqref{critical-points_loci1} -- \eqref{critical-points_loci2} has real solutions $\widetilde{X}_1=(X_1,0)$ and $\widetilde{X}_2=(-X_1,0)$.

\smallskip
Analogously, solving the system \eqref{critical-points_loci1} -- \eqref{critical-points_loci2} for points $(0,y)$ with $y\neq 0$, we obtain
\[ - 45y^4 - 36B_1y^2 + (1 + 9B_1^2) = 0, \]
whose real solutions are
\[ Y_1 = \frac{1}{\sqrt{15}} \sqrt{\sqrt{81 B_1^2 + 5} - 6B_1},\
Y_2 = -\frac{1}{\sqrt{15}} \sqrt{\sqrt{81 B_1^2 + 5} - 6B_1}.
\]
Hence the system \eqref{critical-points_loci1} -- \eqref{critical-points_loci2} has real solutions $\widetilde{Y}_1=(0,Y_1)$ and $\widetilde{Y}_2=(0,-Y_1)$.

\subsubsection[Case: $B_1\neq 0$, $B_2\neq 0$]{Case: $B_1\neq 0$, $B_2\neq 0$}
\label{general_triangles}

Take $m\in \R\setminus \{0\}$ and substitute $y=mx$ in equations \eqref{Ax_Ay} to obtain:
\[
A_x(x,mx) = \frac{2A(x,mx)}{3x(1+m^2)}, \quad
A_y(x,mx) = m\frac{2A(x,mx)}{3x(1+m^2)},
\]
so,
\begin{equation}
 A_y(x,mx) = m A_x(x,mx).
\label{Ay(x,mx)}
\end{equation}

Setting $y=mx$ in the system \eqref{critical-points_loci1} -- \eqref{critical-points_loci2} we obtain:
\begin{align*}
2xA(x,mx) - 3x^2(1+m^2)A_x(x,mx) &= 0, \\
2mxA(x,mx) - 3x^2(1+m^2)A_y(x,mx) &= 0.
\end{align*}

Hence, if $x\neq 0$ and $m\neq 0$, the system \eqref{critical-points_loci1} -- \eqref{critical-points_loci2} is equivalent to the single equation:
\begin{equation}
2A(x,mx) - 3x(1+m^2)A_x(x,mx) = 0.
\label{Sistem_one_equation}
\end{equation}

Since $A_x(x,y) = 4( x^3 + xy^2 - B_1x -B_2y )$, it follows that:
\[ A_x(x,mx) = 4x^3(1 + m^2) - 4x(B_1 + B_2m). \]
Also,
\[
A(x,mx) = (1/9) + (1 + m^2)^2x^4 - 2x^2[B_1 + 2B_2m - B_1m^2] + (B_1^2 + B_2^2).
\]
After substituting these last identities into \eqref{Sistem_one_equation}, we conclude it is equivalent to:
\begin{equation}
 45(1 + m^2)^2x^4 - 18(2B_1 + B_2m + 4B_1m^2 + 3B_2m^3)x^2 - \big(9\abs{\mathtt{B}_{\tr}}^2 + 1\big) = 0.
\label{simplified_system}
\end{equation}

So, for $y=mx$ the equation $B_2m^2+2B_1m-B_2=0$ holds. Since $B_2\neq 0$,
\[ m = \frac{-B_1 \pm \sqrt{B_1^2+B_2^2}}{B_2}. \]
If $m_1 = \frac{-B_1 + \abs{\mathtt{B}_{\tr}}}{B_2}$ and $m_2 = \frac{-B_1 - \abs{\mathtt{B}_{\tr}}}{B_2}$,
\[
m_1m_2 = \frac{\big(-B_1+\sqrt{B_1^2+B_2^2}\big)\big(-B_1-\sqrt{B_1^2+B_2^2}\big)}{B_2^2} = -1.
\]
Hence, the angle between the two lines in the degenerated conic is $\pi/2$.

We will use the following identities:
\begin{equation}
1 + m_1^2 = 2 \frac{\abs{\mathtt{B}_{\tr}}^2 - B_1\abs{\mathtt{B}_{\tr}}}{B_2^2}
= 2\frac{\abs{\mathtt{B}_{\tr}}\big( \abs{\mathtt{B}_{\tr}} - B_1\big)}{B_2^2},
\label{Eq1}
\end{equation}
\begin{equation}
 1 + m_2^2 = 2 \frac{\abs{\mathtt{B}_{\tr}}^2 + B_1\abs{\mathtt{B}_{\tr}}}{B_2^2}
= 2\frac{\abs{\mathtt{B}_{\tr}}\big( \abs{\mathtt{B}_{\tr}} + B_1\big)}{B_2^2},
\label{Eq2}
\end{equation}

\begin{equation}
2B_1 + B_2m_1 + 4B_1m_1^2 + 3B_2m_1^3 = \frac{- 4B_1\abs{\mathtt{B}_{\tr}}^2 + 4\abs{\mathtt{B}_{\tr}}^3}{B_2^2},
\label{Eq3}
\end{equation}
\begin{equation}
2B_1 + B_2m_2 + 4B_1m_2^2 + 3B_2m_2^3 = \frac{- 4B_1\abs{\mathtt{B}_{\tr}}^2 - 4\abs{\mathtt{B}_{\tr}}^3}{B_2^2}.
\label{Eq4}
\end{equation}

Using equations \eqref{Eq1} and \eqref{Eq3}, the equation \eqref{simplified_system} for $m_1$ is:
\begin{align*}
& 180 \abs{\mathtt{B}_{\tr}}^2\big( \abs{\mathtt{B}_{\tr}} - B_1\big)^2 x^4 - 72 B_2^2 \abs{\mathtt{B}_{\tr}}^2 \big( \abs{\mathtt{B}_{\tr}} -
B_1 \big) x^2 \\
&\quad - B_2^4 \big(9\abs{\mathtt{B}_{\tr}}^2 + 1\big)
= 0,
\end{align*}
with real solutions
\[
x_1, x_2 = \pm \frac{B_2}{\sqrt{30}} \frac{\sqrt{ \sqrt{ 81 \abs{\mathtt{B}_{\tr}}^2 + 5} + 6 \abs{\mathtt{B}_{\tr}} }}
{\sqrt{\abs{\mathtt{B}_{\tr}}^2 - B_1 \abs{\mathtt{B}_{\tr}}}},
\]
where $\abs{\mathtt{B}_{\tr}}^2 - B_1\abs{\mathtt{B}_{\tr}} \geq 0$. If we set
\[ x_0 = \frac{\sqrt{ \sqrt{ 81 \abs{\mathtt{B}_{\tr}}^2 + 5} + 6 \abs{\mathtt{B}_{\tr}} }}
{\sqrt{\abs{\mathtt{B}_{\tr}}^2 - B_1 \abs{\mathtt{B}_{\tr}}}} \]
from these solutions, we obtain two critical points $\widetilde{Z}_1$ and $\widetilde{Z}_2$ for the curvature $K$:
\begin{align*}
\widetilde{Z}_1 &= \Bigg( \frac{B_2}{\sqrt{30}} x_0,
\frac{(-B_1 + \abs{\mathtt{B}_{\tr}})}{\sqrt{30}} x_0 \Bigg), \\
\widetilde{Z}_2 &= - \widetilde{Z}_1.
\end{align*}

Analogously, from equations \eqref{Eq2} and \eqref{Eq4} we obtain the equation \eqref{simplified_system} for $m_2$ is:
\begin{align*}
& 180 \abs{\mathtt{B}_{\tr}}^2\big( \abs{\mathtt{B}_{\tr}} + B_1\big)^2 x^4 + 72 B_2^2 \abs{\mathtt{B}_{\tr}}^2 \big( \abs{\mathtt{B}_{\tr}} + B_1 \big) x^2  \\
& \quad - B_2^4 \big(9\abs{\mathtt{B}_{\tr}}^2 + 1\big) = 0,
\end{align*}
with real solutions
\[
x_1, x_2 = \pm \frac{B_2}{\sqrt{30}} \frac{\sqrt{ \sqrt{ 81 \abs{\mathtt{B}_{\tr}}^2 + 5} - 6 \abs{\mathtt{B}_{\tr}} }}
{\sqrt{ \abs{\mathtt{B}_{\tr}}^2 + B_1 \abs{\mathtt{B}_{\tr}} }}.
\]
Note that
\[
- 6 \abs{\mathtt{B}_{\tr}} + \sqrt{81 \, \abs{\mathtt{B}_{\tr}}^{2} + 5} > 0 \; \iff \; 7\abs{\mathtt{B}_{\tr}}^2 + 1 > 0.
\]

If we set
\[ x'_0 = \frac{\sqrt{ \sqrt{ 81 \abs{\mathtt{B}_{\tr}}^2 + 5} - 6 \abs{\mathtt{B}_{\tr}} }}
{\sqrt{\abs{\mathtt{B}_{\tr}}^2 + B_1 \abs{\mathtt{B}_{\tr}}}}, \]
we obtain two critical points $\widetilde{W}_1$ and $\widetilde{W}_2$ for the curvature $K$ from these solutions:
\begin{align*}
\widetilde{W}_1 &= \Bigg( \frac{B_2}{\sqrt{30}} x'_0,
\frac{(-B_1 - \abs{\mathtt{B}_{\tr}})}{\sqrt{30}} x'_0 \Bigg), \\
\widetilde{W}_2 &= - \widetilde{W}_1.
\end{align*}

The analysis carried out can be subsumed in the following:

\begin{theorem}
\label{critical-points_theorem}
If $f(z)=(z-a)(z-b)(z-c)$, then the curvature $K(x,y)$ associated with the graph of $f$ has one critical point located at $\mathtt{b}_{\tr}$. The other critical points are described as follows:
\begin{enumerate}
\item If $B_1=B_2=0$, there is a circumference of critical points centered at the barycenter, with radius $\frac{1}{\sqrt[4]{45}}$.

\item If $B_1\neq 0$ and $B_2=0$, the other 4 critical points are symmetric, by pairs, with respect to the barycenter:
\[
\mathtt{b}_{\tr}+\widetilde{X}_1, \mathtt{b}_{\tr}-\widetilde{X}_1
\quad \text{ and }\quad
\mathtt{b}_{\tr}+\widetilde{Y}_1, \mathtt{b}_{\tr} - \widetilde{Y}_1.
\]

\item If $\mathtt{B}_\tr\neq0$, the other 4 critical points are symmetric, by pairs, with respect to the barycenter:
\[
\mathtt{b}_{\tr}+\widetilde{Z}_1, \mathtt{b}_{\tr}-\widetilde{Z}_1
\quad \text{ and }\quad
\mathtt{b}_{\tr}+\widetilde{W}_1, \mathtt{b}_{\tr} - \widetilde{W}_1.
\]
\end{enumerate}
\end{theorem}

\subsection[The Nature of Critical Points]{The Nature of Critical Points}
\label{hessians_critical-points}

We assume that $\mathtt{B}_\tr\neq 0$ throughout this section.

Recall that, when $\mathtt{B}_\tr=0$, Theorem~\ref{critical-points_theorem}(1) shows that there exists a critical circumference centered at the barycenter $\mathtt{b}_{\tr}$ of $T$, where the curvature attains a global maximum.

 We describe how $\mathtt{b}_{\tr}$ and $\mathtt{B}_{\tr}$ transform under isometries, using the symmetries of the curvature. We first translate $\tr$ to locate its barycenter at the origin and then rotate it to get $\imag(\mathtt{B}_{\tr})=0$. In this case, the critical points of the curvature (Theorem~\ref{critical-points_theorem}(2)) are located on lines parallel to the coordinate axes, which intersect at the barycenter. In fact, they coincide with the coordinate axes in this case. This simplifies computations determining the nature of four critical points. Using the symmetry of the curvature with respect to the origin, we complete the analysis.

Any triangle $\triangle$ with $\mathtt{B}_\tr\neq 0$ can be translated by $t:=-\mathtt{b}_{\tr}$ to obtain  $\tr_{+t}:=\tr(a+t,b+t,c+t)$. Lemma~\ref{B_invariant} implies that $\mathtt{b}_{\tr_{+t}}=\mathtt{b}_{\tr}-\mathtt{b}_{\tr}=0$, i.e., $\tr_{+t}$ has its barycenter at the origin. The Gaussian curvature has translational symmetry (Lemma~\ref{symmetry_curvature}(a)), which implies that its critical points correspond to the critical points of $K_{\tr_{+t}}$. Write $\mathtt{B}_{\tr_{+t}}=B_1^t+iB_2^t$ and suppose that $B_2^t\neq 0$. Let $\tr_{t\theta}=\tr\big(e^{i\theta}(a+t),e^{i\theta}(b+t),e^{i\theta}(c+t)\big)$. The critical points of $K_{\tr_{+t}}$ correspond to the critical points of $K_{t\theta}$ after rotation because, by Lemma~\ref{symmetry_curvature}(b), the curvature has rotational symmetry. Note also that if $\mathtt{b}_{\tr_{t\theta}}$ is the barycenter of $\tr_{t\theta}$, we have by Lemma~\ref{B_invariant}:
\[
\mathtt{b}_{\tr_{t\theta}} = e^{i\theta}\mathtt{b}_{\tr_{+t}} = e^{i\theta}\cdot 0 = 0.
\]

Moreover, if $\mathtt{B}_{\tr_{t\theta}}=B_1^{t\theta}+iB_2^{t\theta}$, we can choose $\theta$ such that $B_1^{t\theta}\neq 0$, but $B_2^{t\theta}=0$ by Lemma~\ref{symmetry_curvature}(c). We will use the following symmetry:
\begin{align*}
K_{t\theta}(z) &= - C_0 \frac{\abs{z-\mathtt{b}_{\tr_{t\theta}}}^2}{\big[ (1/9) + \abs{(z-\mathtt{b}_{\tr_{t\theta}})^2-\mathtt{B}_{\tr_{t\theta}} }^2 \big]^3}, \\
&= - C_0 \frac{\abs{z}^2}{\big[ (1/9) + \abs{ z^2-\mathtt{B}_{\tr_{t\theta}} }^2 \big]^3},  && \text {by } \mathtt{b}_{\tr_{t\theta}}=0,  \\
&= -C_0\frac{\abs{-z}^2}{\big[ (1/9)+\abs{(-z)^2-\mathtt{B}_{\tr_{t\theta}}}^2 \big]^3},  &&  \\
&= K_{t\theta}(-z).
\end{align*}

\subsubsection*{The Nature of Critical Points}

We may suppose that $B_2=0$, $B_1\neq 0$, $\mathtt{b}_\tr=0$, based on the discussion above. Thus, the following equation holds:
\begin{equation}
\label{eq:K-symmetry}
K(z)=K(-z).
\end{equation}

By Theorem \ref{critical-points_theorem}(2), the critical points of $K$ are located at the coordinate axes, i.e., at $(0,0)$, $(\pm X_1,0)$ and
$(0,\pm Y_1)$, where
\[
X_1 = \frac{1}{\sqrt{15}} \sqrt{\sqrt{81 B_1^2 + 5} + 6B_1}, \quad
Y_1 = \frac{1}{\sqrt{15}} \sqrt{\sqrt{81 B_1^2 + 5} - 6B_1}.
\]

Since $B_2=0$ and $B_1\neq 0$ we have
\[ A(x,y) = \frac{1}{9} + (x^2-y^2-B_1)^2 + 4x^2y^2. \]
By \eqref{Kx} and \eqref{Ky}, the partial derivatives $K_x$ and $K_y$ are given by
\[ K_x(x,y) = \frac{-C_0}{A(x,y)^4} \bigg[ 2xA(x,y)-3(x^2+y^2)A_x(x,y) \bigg], \]
\[ K_y(x,y) = \frac{-C_0}{A(x,y)^4} \bigg[ 2yA(x,y)-3(x^2+y^2)A_y(x,y) \bigg] , \]
where $C_0=\frac{2^3}{3^4}$. Also,
\[ K_{xx}(x,y) = - \frac{C_0}{A(x,y)^5}\ N_{xx}(x,y), \quad
K_{xy}(x,y) = - \frac{C_0}{A(x,y)^5}\ N_{xy}(x,y), \]
and
\[ K_{yy}(x,y) = - \frac{C_0}{A(x,y)^5}\ N_{yy}(x,y), \]
where
\begin{align*}
N_{xx}(x,y) &:= 2A(x,y)^2 - 12xA(x,y)A_x(x,y) + 12(x^2+y^2)A_x(x,y)^2 \\
&\quad - 3(x^2+y^2)A(x,y)A_{xx}(x,y),
\end{align*}
\begin{align*}
N_{yy}(x,y) &:= 2A(x,y)^2 - 12yA(x,y)A_y(x,y) + 12(x^2+y^2)A_y(x,y)^2 \\
&\quad - 3(x^2+y^2)A(x,y)A_{yy}(x,y),
\end{align*}
and
\begin{align*}
N_{xy}(x,y) &:= - 6yA(x,y)A_x(x,y) - 6xA(x,y)A_y(x,y) \\
&\quad + 12(x^2+y^2)A_x(x,y)A_y(x,y) - 3(x^2+y^2)A(x,y)A_{xy}(x,y).
\end{align*}

Hence, we define:
\begin{equation}
\label{HessN}
\text{Hess}(N(x,y)) := N_{xx}(x,y)N_{yy}(x,y) - N_{xy}(x,y)^2.
\end{equation}

The following relation holds:
\begin{equation}
\label{HessK}
\text{Hess}(K(x,y)) = \frac{C_0^2}{A(x,y)^{10}} \,\Hess(N(x,y)).
\end{equation}

The following expressions were obtained using SageMath \cite{Sage}:
\begin{equation}
\label{HessK(x,0)}
\text{Hess}(K(x,0)) = - C_1 C_0^2 \frac{P_1(x)\cdot P_2(x)}{P_3(x)^9},
\end{equation}
\begin{equation}
\label{HessK(0,y)}
\text{Hess}(K(0,y)) = - C_1 C_0^2 \frac{Q_1(y)\cdot Q_2(y)}{Q_3(y)^9},
\end{equation}
where
\begin{itemize}
\item $P_1(x) = 4455x^8 - 6642B_1x^6 + 2106B_1^3x^2 + 81B_1^4 - 360x^4 + 234B_1x^2 + 18B_1^2 + 1$,
\item $P_2(x) = 45x^4 + 72B_1x^2 - 9B_1^2 - 1$,
\item $P_3(x) = 9x^4 - 18B_1x^2 + 9B_1^2 + 1$,
\item $Q_1(y) = 4455y^8 + 6642B_1y^6 - 2106B_1^3y^2 + 81B_1^4 - 360y^4 - 234B_1 y^2 + 18 B_1^2 + 1$,
\item $Q_2(y) = 45y^4 - 72B_1y^2 - 9B_1^2 - 1$,
\item $Q_3(y) = 9y^{4} + 18B_1y^2 + 9B_1^2 + 1$, and,
\item $C_1=2125764$.
\end{itemize}

\smallskip
We now state two important results regarding $\text{Hess}(K(X,Y))$; the proof can be found in the Appendix.

\begin{lemma}
\label{sign_Hess}
From identities \eqref{HessK(x,0)} and \eqref{HessK(0,y)}, it follows that:
\begin{enumerate}
\item[(a)] If $B_1>0$, then $P_1(X_1)<0$, $P_2(X_1)>0$ and $P_3(X_1)>0$. Also,
$Q_1(Y_1)<0$, $Q_2(Y_1)<0$, and $Q_3(Y_1)>0$. Consequently,
\[\Hess(K(X_1,0))>0, \quad\Hess(K(0,Y_1))<0. \]
\item[(b)] If $B_1<0$, then $P_1(X_1)<0$, $P_2(X_1)<0$, and $P_3(X_1)>0$. Also, $Q_1(Y_1)<0$, $Q_2(Y_1)>0$, $Q_3(Y_1)>0$. Consequently,
\[\Hess(K(X_1,0))<0, \quad\Hess(K(0,Y_1))>0. \]
\end{enumerate}
\end{lemma}

\begin{lemma}
\label{sign_Kxx}
\begin{enumerate}
\item[(a)] If $B_1>0$, then $K_{xx}(X_1,0)>0$.
\item[(b)] If $B_1<0$, then $K_{xx}(0,Y_1)>0$.
\end{enumerate}
\end{lemma}

To summarize,
\begin{enumerate}[I.]
\item ($B_1>0$) By Lemma \ref{sign_Hess}(a), $\text{Hess}(K(X_1,0))>0$ and $\text{Hess}(K(0,Y_1))<0$. Hence,
\begin{itemize}
\item $(0,Y_1)$ is a saddle point.
\item By Lemma \ref{sign_Kxx}(a), $K_{xx}(X_1,0)>0$. Thus, $(X_1,0)$ is a local minimum.
\item The symmetry \eqref{eq:K-symmetry} implies that $(0,-Y_1)$ is a saddle  point and $(-X_1,0)$ is a local minimum.
\end{itemize}

\item ($B_1<0$) By Lemma \ref{sign_Hess}(b), $\text{Hess}(K(X_1,0))<0$ and $\text{Hess}(K(0,Y_1))>0$. Hence,
\begin{itemize}
\item $(X_1,0)$ is a saddle point.
\item By Lemma \ref{sign_Kxx}(b), $K_{xx}(0,Y_1)>0$. Therefore, $(0,Y_1)$ is a local minimum.
\item The symmetry \eqref{eq:K-symmetry} implies that $(0,-Y_1)$ is a local minimum and $(-X_1,0)$ is a saddle point.
\end{itemize}
\end{enumerate}
\smallskip

The following identities hold at the barycenter $(0,0)$:
\[
\text{Hess}(K(0,0)) = \frac{C_1 C_0^2}{(9B_1^2 + 1)^6}>0, \quad
K_{xx}(0,0) = - \frac{1458 C_0}{(9B_1^2 + 1)^3}<0.
\]
Hence $(0,0)$ is a global maximum. The analysis above is summarized in the next:

\begin{theorem}
\label{Hessiano_B2=0}
Suppose that $B_2=0$, $B_1\neq 0$ and $\mathtt{b}_\tr=0$ hold for $T=\tr(a,b,c)$. The critical points of the Gaussian curvature $K$ associated with the cubic polynomial $f(z)=(z-a)(z-b)(z-c)$ admit the following classification:
\begin{enumerate}[(a)]
\item[(a)] The origin $(0,0)$ is a global maximum.
\item[(b)] If $B_1>0$, then $(\pm X_1,0)$ are local minima and $(0,\pm Y_1)$ are saddle points.
\item[(c)] If $B_1<0$, then $(\pm X_1,0)$ are saddle points and $(0,\pm Y_1)$ are local minima.
\end{enumerate}
\end{theorem}

More generally, we have the following:
\begin{theorem}
\label{type_critical-points}
Let $T=\tr(a,b,c)$ be such that $\mathtt{B}_\tr \neq 0$. There exist two orthogonal lines $L_1$ and $L_2$ in the plane such that $L_1\cap L_2=\mathtt{b}_\tr$. The Gaussian curvature $K$ has two local minima located at $L_1$, which are equidistant from the barycenter $\mathtt{b}_\tr$, two saddle points located at $L_2$, which are also equidistant from the barycenter, and a global maximum at the barycenter.
\end{theorem}

Geometrically, the critical points of the Gaussian curvature $K$ determine a rhombus whose diagonals intersect at the barycenter of the triangle.

\section[Critical Quadrilaterals and the Siebeck-Marden Theorem]{Critical Quadrilaterals and the Siebeck-Marden Theorem}
\label{critical_quadrilaterals}

In this section, we analyze geometrical configurations of quadrilaterals determined by the critical points of the Gaussian curvature $K$ as prescribed by Theorem~\ref{type_critical-points}. We will call these \textsf{critical quadrilaterals}.

\subsection[Critical Quadrilaterals]{Critical Quadrilaterals}
\label{critical-quadrilaterals}

In the case $B_2=0,B_1\neq 0$, the critical points determine a quadrilateral $\mathfrak{Q}_\tr$ whose diagonals are parallel to the coordinate axes and they intersect at the barycenter of $T$. The diagonals of such a quadrilateral are:
\begin{equation}\label{critical_diagonals_B2=0}
\begin{aligned}
d_1 &=\frac{2}{\sqrt{15}} \sqrt{\sqrt{81 B_1^2 + 5} + 6B_1}, \\
d_2 &=\frac{2}{\sqrt{15}} \sqrt{\sqrt{81 B_1^2 + 5} - 6B_1}.
\end{aligned}
\end{equation}

More generally: when $\mathtt{B}_{\tr}\neq 0$, the critical points determine a quadrilateral $\mathfrak{Q}_\tr$ whose diagonals form an angle of $\pi/2$ and intersect at the barycenter $\mathtt{b}_{\tr}$ of $T$. Moreover, $\mathfrak{Q}_\tr$ is a quadrilateral with diagonals of lengths $d_1$ and $d_2$, and sides of length $\ell_{\mathfrak{Q}_\tr}$ given by:
\begin{equation}\label{critical_diagonals}
\begin{aligned}
d_1 &=\frac{2}{\sqrt{15}} \sqrt{\sqrt{81\abs{\mathtt{B}_{\tr}}^2+5}
+6\abs{\mathtt{B}_{\tr}}}, \\
d_2 &=\frac{2}{\sqrt{15}} \sqrt{\sqrt{81\abs{\mathtt{B}_{\tr}}^2+5}
-6\abs{\mathtt{B}_{\tr}}},
\end{aligned}
\end{equation}
\[ \ell_{\mathfrak{Q}_\tr}=\frac{\sqrt{2}}{\sqrt{15}} \sqrt[4]{ 81 \abs{\mathtt{B}_{\tr}}^2 + 5 }. \]

\subsection[Invariance Properties]{Invariance Properties}
\label{invariance-properties_quadrilaterals}

We study here invariant properties of critical quadrilaterals, specifically, their invariance under translations and rotations.

Since $\mathtt{b}_{\tr_{+t}}=\mathtt{b}_{\tr}+t$, Theorem~\ref{critical-points_theorem} and Lemma~\ref{B_invariant} imply
\[ \mathfrak{Q}_{\tr_{+t}} = t + \mathfrak{Q}_{\tr}. \]
That is, the translated critical quadrilateral is the original critical quadrilateral translated by $t$.

Suppose $\mathtt{B}_{\tr}\neq 0$. If $d_1,d_2$ are the diagonals of $\mathfrak{Q}_{\tr}$ and $d_1^\theta,d_2^\theta$ are the diagonals of $\mathfrak{Q}_{\tr_\theta}$, Lemma~\ref{B_invariant} implies that
\[
\abs{\mathtt{B}_{\tr_\theta}} = \abs{e^{2i\theta} \mathtt{B}_{\tr}} = \abs{\mathtt{B}_{\tr}}.
\]

By \eqref{critical_diagonals}, it follows that:
\[  d_1^\theta = d_1, \quad d_2^\theta = d_2. \]
That is, the quadrilaterals $\mathfrak{Q}_{\tr}$ and $\mathfrak{Q}_{\tr_\theta}$ are congruent. Hence, we have the following:

\begin{theorem}
\label{Asignation_triangles-rhombus}
If $T=\tr(a,b,c)$ and $T'=\tr(a',b',c')$ are two congruent triangles such that $\mathtt{B}_\tr \neq 0$, then the critical quadrilaterals $\mathfrak{Q}_\tr$ and $\mathfrak{Q}_{\tr'}$ are congruent.
\end{theorem}

\subsection[Critical Quadrilaterals and the Siebeck-Marden Theorem]{Critical Quadrilaterals and the Siebeck-Marden Theorem}
\label{critical-quadrilaterals_Siebeck-Marden}

According to the Siebeck-Marden Theorem, the roots
$z_1=\mathtt{b}_{\tr} + \sqrt{\mathtt{B}_{\tr}}$ and $z_2=\mathtt{b}_{\tr} - \sqrt{\mathtt{B}_{\tr}}$ of $f'(z)$ are the foci of the ellipse that is tangent to the midpoints of the sides of $T$. The zero of the second derivative $z=\mathtt{b}_{\tr}$ is the midpoint of the segment $[z_1,z_2]$, 
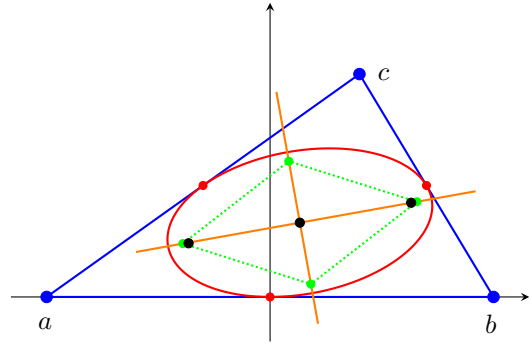
\begin{figure}[!htbp]
\centering
\scalebox{1}{%
\begin{tikzpicture}
\begin{axis}[
    axis equal image,
    xmin=-1.45, xmax=1.45,
    ymin=-0.25, ymax=1.65,
    axis lines=center,
    xlabel={}, ylabel={},
    ticks=none,
]
\addplot[thick,blue] coordinates {
    (-1.25,0)
    (1.25,0)
    (0.5,1.25)
    (-1.25,0)
};
\addplot[thick,solid,color=red,domain=0:360,samples=300]
    ({0.16667 + 0.74904*cos(x)*cos(10.16157) - 0.40145*sin(x)*sin(10.16157)},
     {0.41667 + 0.74904*cos(x)*sin(10.16157) + 0.40145*sin(x)*cos(10.16157)});
\begin{scope}[
    shift={(axis cs:0,0)},
    x={(axis direction cs:1,0)},
    y={(axis direction cs:0,1)}]
\draw[blue,fill=blue] (-1.25,0) circle (.5ex);
\draw[blue,fill=blue] (1.25,0) circle (.5ex);
\draw[blue,fill=blue] (0.5,1.25) circle (.5ex);
\draw (-1.35,-0.15) node[right,black] {$a$};
\draw (1.15,-0.15) node[right,black] {$b$};
\draw (0.55,1.25) node[right,black] {$c$};
\draw[thick,color=orange] (-0.75,0.2522) -- (1.15,0.5929);
\draw[thick,color=orange] (0.26822,-0.15) -- (0.0352,1.15);
\draw[thick,color=green,densely dotted] (0.8213,0.5340) -- (0.2284,0.07213);
\draw[thick,color=green,densely dotted] (0.8213,0.5340) -- (0.1049,0.7612);
\draw[thick,color=green,densely dotted] (-0.4880,0.2993) -- (0.2284,0.07213);
\draw[thick,color=green,densely dotted] (-0.4880,0.2993) -- (0.1049,0.7612);
\draw[green,fill=green] (0.8213,0.5340) circle (.35ex);
\draw[green,fill=green] (-0.4880,0.2993) circle (.35ex);
\draw[green,fill=green] (0.2284,0.07213) circle (.35ex);
\draw[green,fill=green] (0.1049,0.7612) circle (.35ex);
\draw[black,fill=black] (0.7892,0.5283) circle (.4ex);
\draw[black,fill=black] (-0.4559,0.301) circle (.4ex);
\draw[red,fill=red] (0,0) circle (.35ex);
\draw[red,fill=red] (-0.375,0.625) circle (.35ex);
\draw[red,fill=red] (0.875,0.625) circle (.35ex);
\draw[black,fill=black] (0.1666,0.4166) circle (.4ex);
\end{scope}
\end{axis}
\end{tikzpicture}
}
\caption{Siebeck-Marden ellipse associated to $T$}
\label{fig:S-M_ellipse}
\end{figure}

\subsubsection[Concentric Circumferences]{Concentric Circumferences}
\label{concentric_circumferences}

When triangle $T$ is equilateral with side length $\ell$, its Siebeck-Marden ellipse is a circumference with center $\mathtt{b}_{\tr}$ and radius $\sqrt{3}\ell/6$, and the critical circumference has center  $\mathtt{b}_{\tr}$ and radius $1/\sqrt[4]{45}$. Therefore, these two circumferences are concentric.

Figure \ref{fig:equilateral}(A) shows an equilateral triangle ($B_\tr=0$) in blue, with the critical circumference in green. The Siebeck-Marden ellipse for this triangle is depicted in red in Figure~\ref{fig:equilateral}(B) (a circumference in this case), the black dot corresponding to the barycenter. The circumferences are concentric (Theorem~\ref{critical-points_theorem}(1)).
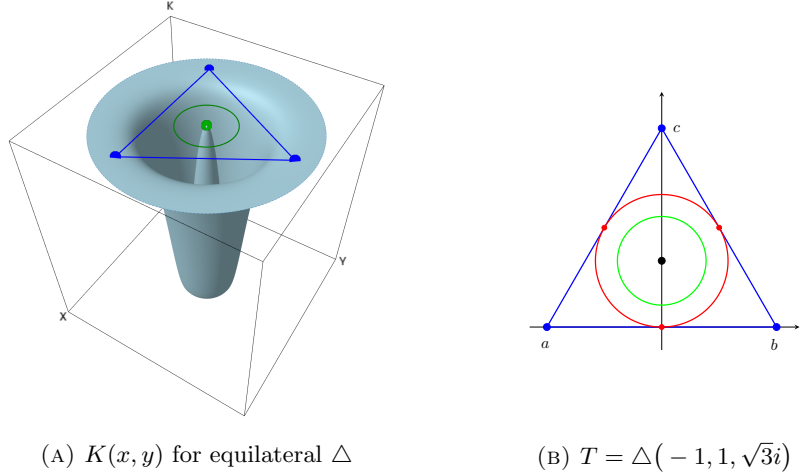
\begin{figure}[!htbp]
\centering
\begin{subfigure}[t]{0.48\textwidth}
\centering
\includegraphics[scale=0.14]{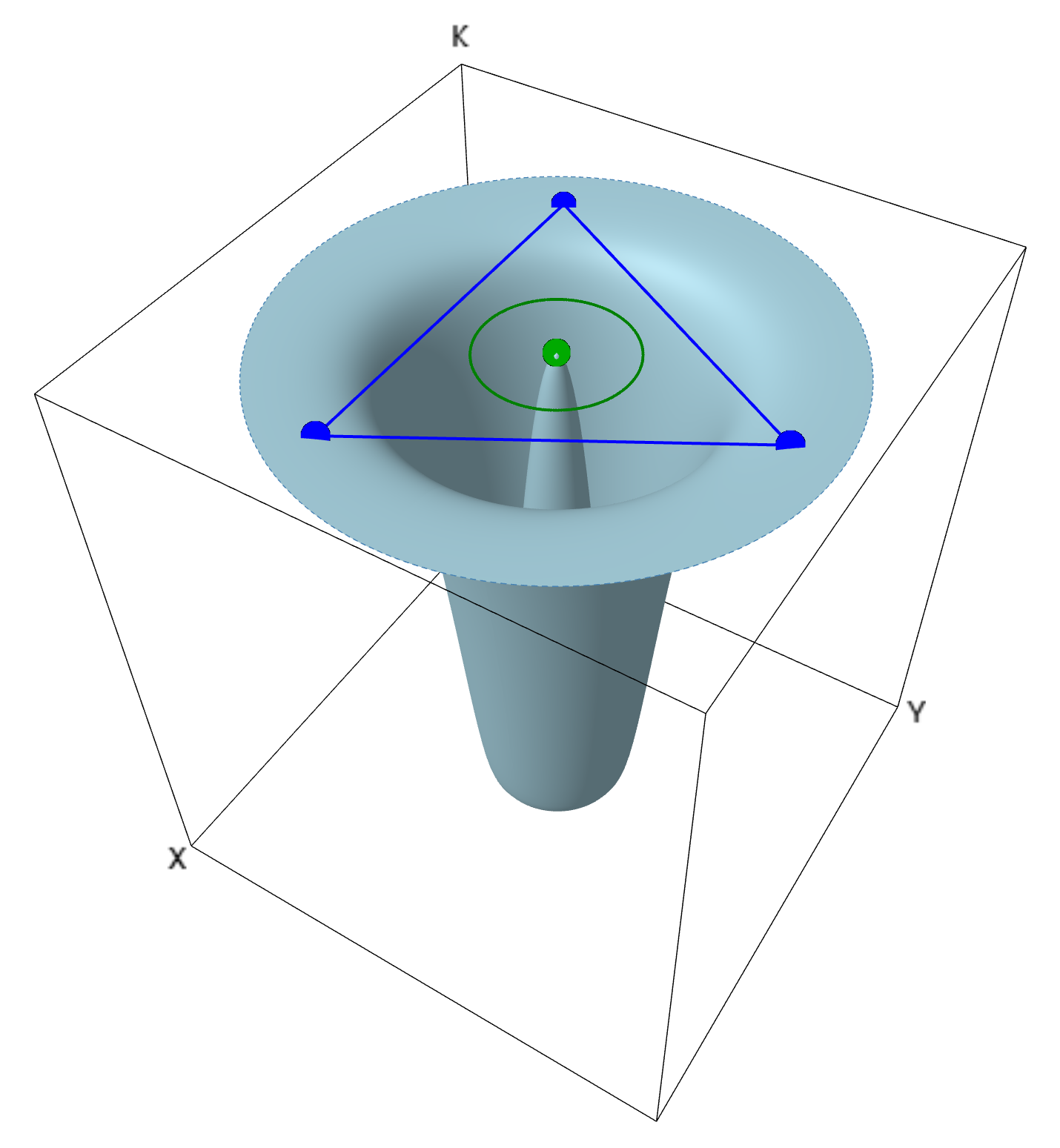}
\caption{$K(x,y)$ for equilateral $\tr$}
\end{subfigure}
\begin{subfigure}[t]{0.48\textwidth}
\centering
\raisebox{1cm}{
\scalebox{0.6}{
\begin{tikzpicture}
\begin{axis}[
    axis equal image,
    xmin=-1.15, xmax=1.2,
    ymin=-0.2, ymax=2.05,
    axis lines=center,
    xlabel={}, ylabel={},
    ticks=none,
]
\addplot[thick,blue] coordinates {
    (-1,0)
    (1,0)
    (0,1.732)
    (-1,0)
};
\addplot[thick,solid,color=red,domain=0:360,samples=300]
    ({0.57735*cos(x)},
     {0.57735 + 0.57735*sin(x)});
\addplot[thick,solid,color=green,domain=0:360,samples=300]
    ({0.386*cos(x)},
     {0.5773 + 0.386*sin(x)});
\begin{scope}[
    shift={(axis cs:0,0)},
    x={(axis direction cs:1,0)},
    y={(axis direction cs:0,1)}]
\draw[blue,fill=blue] (-1,0) circle (.5ex);
\draw[blue,fill=blue] (1,0) circle (.5ex);
\draw[blue,fill=blue] (0,1.732) circle (.5ex);
\draw (-1.1,-0.15) node[right,black] {$a$};
\draw (0.9,-0.15) node[right,black] {$b$};
\draw (0.05,1.732) node[right,black] {$c$};
\draw[red,fill=red] (0,0) circle (.35ex);
\draw[red,fill=red] (0.5,0.8660) circle (.35ex);
\draw[red,fill=red] (-0.5,0.8660) circle (.35ex);
\draw[black,fill=black] (0,0.5773) circle (.5ex);
\end{scope}
\end{axis}
\end{tikzpicture}
}
}
\caption{$T=\triangle\big(-1,1,\sqrt{3}i\big)$}
\end{subfigure}
\caption{Curvature, critical circumference, and S-M ellipse for an equilateral triangle}
\label{fig:equilateral}
\end{figure}

\subsubsection[Case $B_2=0$]{Case $B_2=0$}
\label{concentric_ellipses}

If $B_1>0$, the Siebeck-Marden ellipse has foci at $z_1$ and $z_2$, and the critical points $\mathtt{b}_{\tr}+\widetilde{X}_1$ and $\mathtt{b}_{\tr}+\widetilde{X}_2$ of the Gaussian curvature are located on the major axis, while the other two critical points, $\mathtt{b}_{\tr}+\widetilde{Y}_1$ and $\mathtt{b}_{\tr}+\widetilde{Y}_2$, are located on the minor axis. The center of this ellipse coincides with the barycenter of $T$, where the Gaussian curvature attains its maximum.

According to the analysis in Section~\ref{critical-points_analysis}, the critical points of the Gaussian curvature are located on lines parallel to the coordinate axes, with a maximum at $\mathtt{b}_{\tr}$. Consequently, the major and minor axes of the Siebeck--Marden ellipse are parallel to the coordinate axes.

Figure \ref{fig:isosceles}(A) shows the graph of the Gaussian curvature and its five critical points: the barycenter, two saddle points, and two minima. Figure~\ref{fig:isosceles}(B) depicts the corresponding isosceles triangle in blue, for which $B_2=0$ and $B_1\neq 0$. It also shows the vertices of the critical quadrilateral, joined by green dotted lines, and the S-M ellipse in red, with its foci and center marked by black dots and its axes shown in orange. Note that two of the critical points, shown in green, lie on the focal axis, in agreement with Theorem~\ref{coincidence_diagonals}.
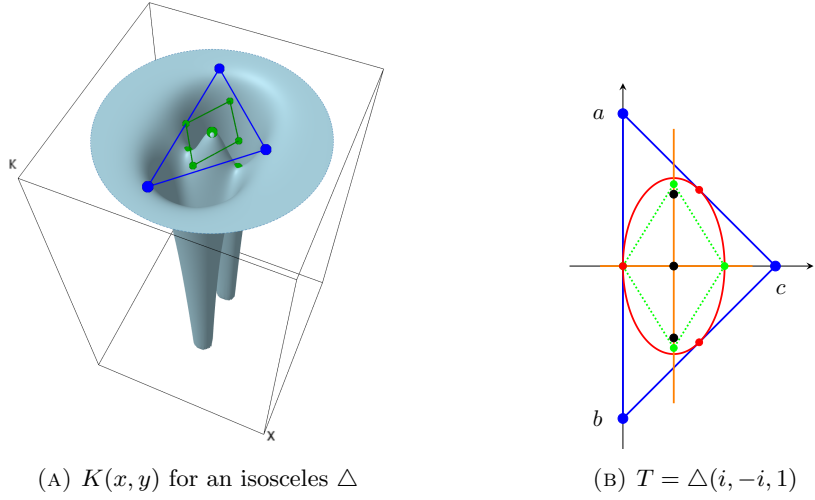
\begin{figure}[!htbp]
\centering
\begin{subfigure}[b]{0.48\textwidth}
\centering
\vspace{0pt}
\includegraphics[scale=0.16]{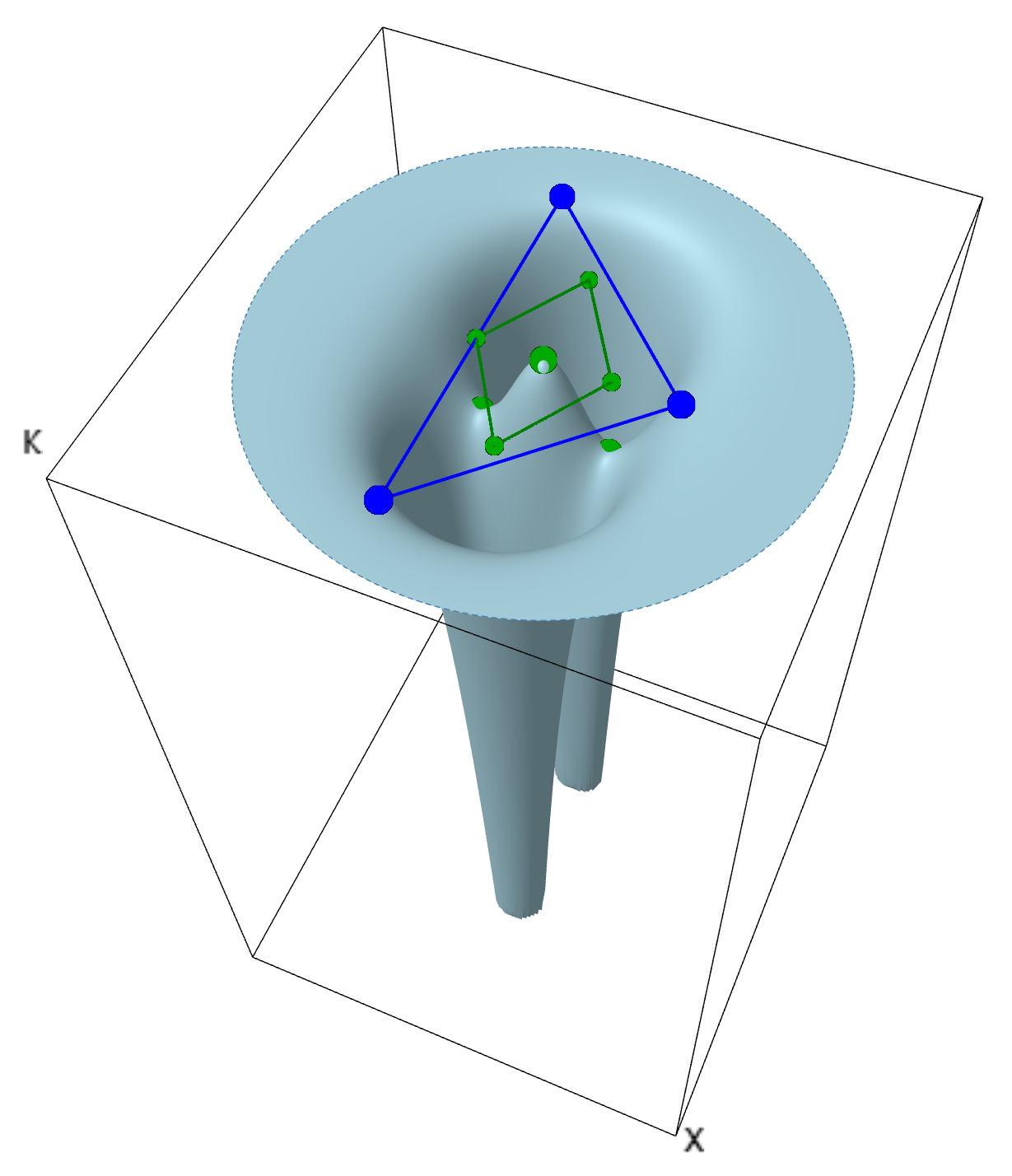}
\caption{$K(x,y)$ for an isosceles $\tr$}
\end{subfigure}
\hfill
\begin{subfigure}[b]{0.48\textwidth}
    \centering
    \vspace{0pt}
    \scalebox{0.85}{%
\begin{tikzpicture}
\begin{axis}[
    axis equal image,
    xmin=-0.35, xmax=1.25,
    ymin=-1.2, ymax=1.2,
    axis lines=center,
    xlabel={}, ylabel={},
    ticks=none,
]
\addplot[thick,blue] coordinates {
    (0,-1)
    (0,1)
    (1,0)
    (0,-1)
};
\addplot[thick,solid,color=red,domain=0:360,samples=300]
    ({0.33333 + 0.57735*cos(x)*cos(90) - 0.33333*sin(x)*sin(90)},
     {0.57735*cos(x)*sin(90) + 0.33333*sin(x)*cos(90)});
\begin{scope}[
    shift={(axis cs:0,0)},
    x={(axis direction cs:1,0)},
    y={(axis direction cs:0,1)}]
\draw[blue,fill=blue] (0,-1) circle (.5ex);
\draw[blue,fill=blue] (0,1) circle (.5ex);
\draw[blue,fill=blue] (1,0) circle (.5ex);
\draw (-0.25,1) node[right,black] {$a$};
\draw (-0.25,-1) node[right,black] {$b$};
\draw (0.95,-0.15) node[right,black] {$c$};
\draw[thick,color=orange] (-0.15,0) -- (0.85,0);
\draw[thick,color=orange] (0.3333,0.9) -- (0.3333,-0.9);
\draw[thick,color=green,densely dotted] (0,0) -- (0.3333,-0.5374);
\draw[thick,color=green,densely dotted] (0,0) -- (0.3333,0.5374);
\draw[thick,color=green,densely dotted] (0.6666,0) -- (0.3333,-0.5374);
\draw[thick,color=green,densely dotted] (0.6666,0) -- (0.3333,0.5374);
\draw[green,fill=green] (0,0) circle (.35ex);
\draw[green,fill=green] (0.6666,0) circle (.35ex);
\draw[green,fill=green] (0.3333,-0.5374) circle (.35ex);
\draw[green,fill=green] (0.3333,0.5374) circle (.35ex);
\draw[black,fill=black] (0.3333,0.4714) circle (.4ex);
\draw[black,fill=black] (0.3333,-0.4714) circle (.4ex);
\draw[red,fill=red] (0,0) circle (.35ex);
\draw[red,fill=red] (0.5,0.5) circle (.35ex);
\draw[red,fill=red] (0.5,-0.5) circle (.35ex);
\draw[black,fill=black] (0.3333,0) circle (.4ex);
\end{scope}
\end{axis}
\end{tikzpicture}
}
\caption{$T=\triangle(i,-i,1)$}
\end{subfigure}
\caption{Curvature, critical points and S-M ellipse for an isosceles triangle}
\label{fig:isosceles}
\end{figure}

\subsubsection[General Case]{General Case}
\label{generic_case}

Let $T=\tr(a,b,c)$ be such that $B_2\neq0$. Recall that $\mathtt{B}_\tr=B_1+iB_2$. By  Theorem~\ref{critical-points_theorem}(3), the slopes $m_1',m_2'$ of the lines through the critical points of the same nature (see Theorem~\ref{type_critical-points}), are given by:
\[
\begin{aligned}
m_1' &= \frac{-B_{1}+\abs{\mathtt{B}_{\tr}}}{B_{2}}, &
m_2' &= \frac{-B_{1}-\abs{\mathtt{B}_{\tr}}}{B_{2}}.
\end{aligned}
\]

The foci $z_1$ and $z_2$ of the Siebeck-Marden ellipse (Eq. \eqref{focal_points}) are:
\[
\begin{aligned}
z_1 &= \mathtt{b}_{\tr} + \sqrt{\mathtt{B}_{\tr}}, &
z_2 &= \mathtt{b}_{\tr} - \sqrt{\mathtt{B}_{\tr}}.
\end{aligned}
\]
Therefore:
\[ \arg(z_1-\mathtt{b}_{\tr}) = \arg\big(\sqrt{\mathtt{B}_{\tr}}\big)
= \tfrac{1}{2} \arg(\mathtt{B}_{\tr}). \]

The focal axis of the Siebeck-Marden ellipse has slope $m_1$ given by:
\begin{align*}
m_1 &= \tan \big( \arg(z_1-\mathtt{b}_{\tr}) \big) \\
&= \tan \big( \tfrac{1}{2}\arg(\mathtt{B}_{\tr}) \big) \\
&= \frac{\sin \big( \arg(\mathtt{B}_{\tr}) \big)}
        {1+\cos \big( \arg(\mathtt{B}_{\tr}) \big)} \\
&= \frac{ \frac{B_2}{ \abs{\mathtt{B}_{\tr}}} }{ 1 + \frac{B_1}{ \abs{\mathtt{B}_{\tr}}} }
= \frac{B_2}{B_1 + \abs{\mathtt{B}_{\tr}} } \\
&= \frac{-B_1+\abs{\mathtt{B}_{\tr}}}{B_2}.
\end{align*}
That is, $m_1'=m_1$ and we have the following:

\begin{theorem}
\label{coincidence_diagonals}
If $f(z)=(z-a)(z-b)(z-c)$ is any cubic polynomial with $\mathtt{B}_\tr\neq 0$, then, the 5 critical points of the Gaussian curvature are located along both axes of the Siebeck-Marden ellipse. In particular, the barycenter corresponds to the center of the ellipse.
\end{theorem}

Figure \ref{fig:scalene}(A) and (B) show a triangle in general position (in blue) with $B_2\neq 0$, and the five critical points of the Gaussian Curvature in green (Theorem~\ref{critical-points_theorem}(3)). The Siebeck-Marden ellipse is in red. Two of the critical points (green) are located on the focal axis as indicated by Theorem~\ref{coincidence_diagonals}. The distance from each focus to its nearest critical point is approximately $0.032$, which is why they appear to coincide in the figure.
\begin{figure}[!htbp]
\centering
\begin{subfigure}[b]{0.45\textwidth}
\centering
\includegraphics[scale=0.175]{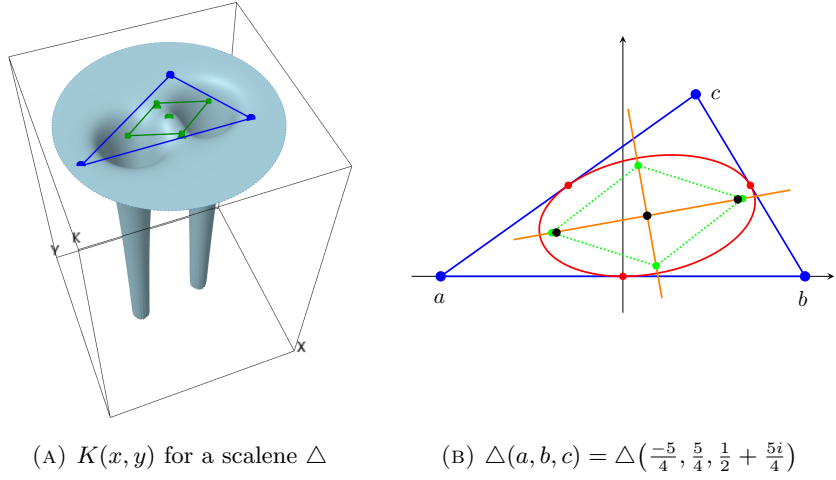}
\caption{$K(x,y)$ for a scalene $\tr$}
\end{subfigure}
\begin{subfigure}[b]{0.45\textwidth}
\centering
\raisebox{1.5cm}{%
\resizebox{\linewidth}{!}{
\begin{tikzpicture}
\begin{axis}[
    axis equal image,
    xmin=-1.45, xmax=1.45,
    ymin=-0.25, ymax=1.65,
    axis lines=center,
    xlabel={}, ylabel={},
    ticks=none,
]
\addplot[thick,blue] coordinates {
    (-1.25,0)
    (1.25,0)
    (0.5,1.25)
    (-1.25,0)
};
\addplot[thick,solid,color=red,domain=0:360,samples=300]
    ({0.16667 + 0.74904*cos(x)*cos(10.16157) - 0.40145*sin(x)*sin(10.16157)},
     {0.41667 + 0.74904*cos(x)*sin(10.16157) + 0.40145*sin(x)*cos(10.16157)});
\begin{scope}[
    shift={(axis cs:0,0)},
    x={(axis direction cs:1,0)},
    y={(axis direction cs:0,1)}]
\draw[blue,fill=blue] (-1.25,0) circle (.5ex);
\draw[blue,fill=blue] (1.25,0) circle (.5ex);
\draw[blue,fill=blue] (0.5,1.25) circle (.5ex);
\draw (-1.35,-0.15) node[right,black] {$a$};
\draw (1.15,-0.15) node[right,black] {$b$};
\draw (0.55,1.25) node[right,black] {$c$};
\draw[thick,color=orange] (-0.75,0.2522) -- (1.15,0.5929);
\draw[thick,color=orange] (0.26822,-0.15) -- (0.0352,1.15);
\draw[thick,color=green,densely dotted] (0.8213,0.5340) -- (0.2284,0.07213);
\draw[thick,color=green,densely dotted] (0.8213,0.5340) -- (0.1049,0.7612);
\draw[thick,color=green,densely dotted] (-0.4880,0.2993) -- (0.2284,0.07213);
\draw[thick,color=green,densely dotted] (-0.4880,0.2993) -- (0.1049,0.7612);
\draw[green,fill=green] (0.8213,0.5340) circle (.35ex);
\draw[green,fill=green] (-0.4880,0.2993) circle (.35ex);
\draw[green,fill=green] (0.2284,0.07213) circle (.35ex);
\draw[green,fill=green] (0.1049,0.7612) circle (.35ex);
\draw[black,fill=black] (0.7892,0.5283) circle (.4ex);
\draw[black,fill=black] (-0.4559,0.301) circle (.4ex);
\draw[red,fill=red] (0,0) circle (.35ex);
\draw[red,fill=red] (-0.375,0.625) circle (.35ex);
\draw[red,fill=red] (0.875,0.625) circle (.35ex);
\draw[black,fill=black] (0.1666,0.4166) circle (.4ex);
\end{scope}
\end{axis}
\end{tikzpicture}%
}%
}%
\caption{$\triangle(a,b,c)=\triangle\big(\tfrac{-5}{4},\tfrac{5}{4},\tfrac{1}{2}+\tfrac{5i}{4}\big)$}
\end{subfigure}%
\caption{Curvature, critical points, and Siebeck-Marden ellipse for a scalene triangle}
\label{fig:scalene}
\end{figure}

\section[The Skew and the Composed Pythagorean Mean]{The Skew and the Composed Pythagorean Mean}
\label{skew_composed-pythagorean-mean}

Hubbard \cite{Hub} studied quasiconformal mappings of the plane using the skew of a given triangle (see also \cite{AHH}). Here, we compare this notion with the quantity $\mathtt{B}_{\tr}$, which has played a fundamental role in this work.

Given the triangle $T=\triangle(a,b,c)$ associated with the cubic polynomial
$f(z)=(z-a)(z-b)(z-c)$, and $A=a+b+c$, $B=ab+ac+bc$ and $C=abc$, we can write
\[ f(z) = z^3-Az^2+Bz-C \]
and we have defined:
\begin{itemize}
\item $\mathtt{b}_\tr=A/3$, the barycenter of $T$.
\item $\mathtt{B}_{\tr}=\mathtt{b}_{\tr}^2 - B/3$.
\end{itemize}

In the next subsection, we give a geometric interpretation of $\mathtt{B}_{\tr}$.

\subsection[The Composed Pythagorean Mean]{The Composed Pythagorean Mean}
\label{composed_Pythagorean-mean}

Recall that the \textsf{harmonic mean} of three nonzero complex numbers $a,b,c$ is given by
\[ H(a,b,c) := \frac{3}{\frac{1}{a}+\frac{1}{b}+\frac{1}{c}}, \]
which can be written as
\[ H(a,b,c) = \frac{3C}{B}. \]
So,
\[ \frac{B}{3} = \frac{G(a,b,c)^3}{H(a,b,c)}, \]
where $G(a,b,c)=\sqrt[3]{C}=\sqrt[3]{abc}$ is the \textsf{geometric mean} of $a,b,c$.

Now, from the definition of $\mathtt{B}_{\tr}$ we obtain
\[ \mathtt{B}_{\tr} = M(a,b,c)^2 - \frac{G(a,b,c)^3}{H(a,b,c)}, \]
where $M(a,b,c)$ is the \textsf{arithmetic mean} of $a,b,c$. We call $\mathtt{B}_{\tr}$ the \textsf{composed Pytha\-gorean mean}, since it contains the three classical notions of Pythagorean means.

\subsection[A Family of Cubic Polynomials with Real $\mathtt{B}_\tr$]{A Family of Cubic Polynomials with Real $\mathtt{B}_\tr$}
\label{family_B2_0}

Consider the family of cubic polynomials $f_{\lambda,t}(z)=(z-it)(z+it)(z-\lambda)$ with $\lambda,t\in \R$ and $t>0$. Associated with this family there is a family of isosceles triangles $T_{\lambda,t}$ in the plane, with vertices $it$, $-it$, and $\lambda$. For this family we have $A=\lambda$ and $B=t^2$, so:
\[ \mathtt{B}_{\tr}=\frac{A^2-3B}{9}=\frac{\lambda^2-3t^2}{9} \in \mathbb{R}. \]
That is, $B_2=0$ and $\mathtt{B}_{\tr}=B_1$ is real.

Critical quadrilaterals, associated with the family of Gaussian curvatures $K_{\lambda,t}$, obtained from $f_{\lambda,t}$, are rhombuses. A rhombus is a simple convex quadrilateral with four equal sides or, equivalently, one whose diagonals are perpendicular bisectors of each other. Two rhombuses $\mathcal{R},\mathcal{R}'$ with diagonals $d_i,d_i',\ (i=1,2)$, respectively, are congruent if, and only if, they have the same diagonals, i.e., $d_1=d_1'$ and $d_2=d_2'$. We also consider the case $f_{0,t}(z)=(z-it)(z+it)$, with degenerate triangles $T_{0,t}$.

Figure \ref{fig:family_B2_0} depicts a $(\lambda,t)$-plane in which each point represents an isosceles triangle $T_{\lambda,t}=\tr(it,-it,\lambda)$, with $t>0$ and $\lambda\in \R$. For example, the triangle in Figure~\ref{fig:isosceles}(B) corresponds to the point $P=(1,1)$ and the value of $\mathtt{B}_{\tr}$ is $-2/9$. The regions where $\mathtt{B}_{\tr}>0$ are shown in orange, and the region where $\mathtt{B}_{\tr}<0$ is shown in green. The black lines $\{\lambda=\pm \sqrt{3}t\}$ correspond to triangles with $\mathtt{B}_{\tr}=0$, i.e., equilateral triangles. The green and orange hyperbolas are level curves of $\mathtt{B}_{\tr}$, i.e., curves with equation $\{\lambda^2-3t^2=9\mathtt{B}_{\tr}\}$ along which $\mathtt{B}_{\tr}$ is constant.

Thus, generic distinct points on a fixed green hyperbola correspond to non-congruent isosceles triangles whose associated rhombuses are congruent (Equations \eqref{critical_diagonals}). The same statement holds for the orange hyperbolas. Any two points in the plane, symmetric with respect to the vertical axis $\{\lambda=0\}$, correspond to congruent triangles.
\begin{figure}[!htbp]
\centering
\begin{tikzpicture}[domain=0:2.5]
\filldraw[thick, draw=yellow!80!black, fill=yellow!80!black, fill opacity=0.4]
  (0,0) -- (4,0) -- (4,2.3094) -- cycle;
\filldraw[thick, draw=yellow!80!black, fill=yellow!80!black, fill opacity=0.4]
  (0,0) -- (-4,0) -- (-4,2.3094) -- cycle;
\filldraw[thick, color=green!80!black, fill opacity=0.4]
  (0,0) -- (4,2.3094) -- (-4,2.3094) -- cycle;

\begin{scope}
  \clip (0,0) -- (4,0) -- (4,2.3094) -- cycle;
  \foreach \k/\col in {0.65/orange!90!black, 2.8/orange!90!black,
                        6.4/orange!90!black,  11.5/orange!90!black}{
    \pgfmathsetmacro{\tmax}{sqrt((16-\k)/3)}
    \draw[thick, color=\col, domain=0:\tmax, samples=150, variable=\t]
      plot ({sqrt(3*\t*\t + \k)}, {\t});
  }
\end{scope}
\begin{scope}
  \clip (0,0) -- (-4,0) -- (-4,2.3094) -- cycle;
  \foreach \k/\col in {0.65/orange!90!black, 2.8/orange!90!black,
                        6.4/orange!90!black,  11.5/orange!90!black}{
    \pgfmathsetmacro{\tmax}{sqrt((16-\k)/3)}
    \draw[thick, color=\col, domain=0:\tmax, samples=150, variable=\t]
      plot ({-sqrt(3*\t*\t + \k)}, {\t});
  }
\end{scope}

\foreach \k/\col in {0.6/green!50!black, 2/green!50!black,
                      5.35/green!50!black,   10.2/green!50!black}{
  \pgfmathsetmacro{\lammax}{min(4, sqrt(3*2.3094*2.3094 - \k))}
  \draw[thick, color=\col,
        domain=-\lammax:\lammax, samples=200, variable=\lam]
    plot ({\lam}, {sqrt(\lam*\lam/3 + \k/3)});
}

\draw[very thick](0,0) -- ( 4, 2.3094);
\draw[very thick](0,0) -- (-4, 2.3094);
\draw[->] (-4.25,0) -- (4.25,0) node[right] {$\lambda$};
\draw[->] (0,-0.5)  -- (0,2.85) node[above] {$t$};
\filldraw[fill=white, draw=black, thick] (0,0) circle (2pt);

\foreach \x in {-4,-3,-2,-1,1,2,3,4}{
  \draw[thick](\x,0.05) -- (\x,-0.05);
}
\foreach \y in {1,2}{
  \draw[thick](0.05,\y) -- (-0.05,\y);
}

\draw(-0.45,-0.3) node[right, black] {$0$};
\foreach \x in {1,2,3,4}{
  \draw(-0.2+\x,-0.3) node[right, black] {$\x$};
}
\draw(-0.35-1,-0.3) node[right, black] {$-1$};
\draw(-0.35-2,-0.3) node[right, black] {$-2$};
\draw(-0.35-3,-0.3) node[right, black] {$-3$};
\draw(-0.35-4,-0.3) node[right, black] {$-4$};

\draw(-0.55,+1) node[right, black] {$1$};
\draw(-0.55,+2) node[right, black] {$2$};

\draw(0.2,2.6) node[right, black] {$\mathtt{B}_{\tr}<0$};
\draw(4.1,1) node[right, black] {$\mathtt{B}_{\tr}>0$};
\draw(-5.4,1) node[right, black] {$\mathtt{B}_{\tr}>0$};
\draw( 3.4,2.55) node[right, black] {$\lambda=\sqrt{3}t$};
\draw(-4.55,2.55) node[right, black] {$\lambda=-\sqrt{3}t$};

\draw[blue,fill=blue] (1,1) circle (.35ex);
\draw(0.75,1.25) node[right, black] {$P$};
\end{tikzpicture}
\caption{Family of isosceles triangles $(\lambda,t) \longleftrightarrow T_{\lambda,t}=\tr(it,-it,\lambda)$}
\label{fig:family_B2_0}
\end{figure}
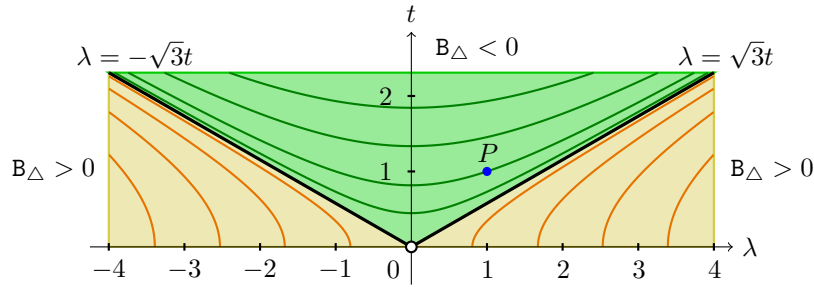

\subsection[The Skew and the Composed Pythagorean Mean]{The Skew and the Composed Pythagorean Mean}
\label{skew_composed-mean}

We close this section with a comparison of the notion of the skew of $T$ with the composed Pythagorean mean. Recall that $T$ is equilateral if and only if $\mathtt{B}_\tr=0$.

The \textsf{skew} of the triangle $T$ measures how far the triangle is from being equilateral. If
$L(T):= \max\{\abs{a-b}:a,b \text{ are vertices of } T\}$ and
$\ell(T):= \min\{\abs{a-b}:a,b \text{ are vertices of } T\}$, define
\[ \mathrm{skew}(T) = \frac{L(T)}{\ell(T)} \]
and note that
\[ T \text{ is equilateral } \quad \Longleftrightarrow \quad \mathrm{skew}(T) = 1. \]

So,
\[
T \text{ is equilateral } \quad \Longleftrightarrow \quad \mathtt{B}_\tr = 0
\quad \Longleftrightarrow \quad \mathrm{skew}(T) = 1.
\]

\begin{example}[The family $f_{\lambda,t}$]
	
For the family $f_{\lambda,t}$ with associated family of triangles $T_{\lambda,t}$, we have
\[
\mathrm{skew}(T_{\lambda,t}) =
\frac{ \max \big\{2t,\sqrt{\lambda^2+t^2} \big\}}{\min \big\{2t,\sqrt{\lambda^2+t^2} \big\} }.
\]

If
\[ \mathrm{skew}(T_{\lambda,t}) = \frac{\sqrt{\lambda^2+t^2}}{2t} \]
then

\begin{align*}
\mathrm{skew}(T_{\lambda,t})^2 - 1 &= \frac{\lambda^2+t^2}{4t^2} - 1 \\
&= 9\frac{\mathtt{B_\tr}}{\ell_3^2}	
\end{align*}
	
Hence,
\[ \mathrm{skew}(T_{\lambda,t})^2 = 1 + 9\frac{\mathtt{B_\tr}}{\ell_3^2} \]
where $\ell_3$ is the smallest side of the triangle. On the other hand, if
\[ \mathrm{skew}(T_{\lambda,t}) = \frac{2t}{\sqrt{\lambda^2+t^2}} \]
then
\begin{align*}
\mathrm{skew}(T_{\lambda,t})^2 - 1 &= \frac{4t^2}{\lambda^2+t^2} - 1 \\
&= 9\frac{-\mathtt{B_\tr}}{\ell_1^2}	
\end{align*}
and therefore,
\[ \mathrm{skew}(T_{\lambda,t})^2 = 1 + 9\frac{-\mathtt{B_\tr}}{\ell_1^2} \]
where $\ell_1$ is the smallest side of the triangle.
	
In either case,
\[ \mathrm{skew}(T_{\lambda,t})^2 = 1 + C(\ell)\abs{\mathtt{B_\tr}} \]
where $\displaystyle C(\ell)=\frac{9}{\ell^2}$, with $\ell$ the smallest side of the triangle.
\end{example}

The study presented in the last section allows us to conjecture:

\begin{conjecture}[The Skew Conjecture]
\label{skew_conjecture}	
Given a cubic polynomial $f(z)$ with nondegenerate triangle of zeros $T=\tr(a,b,c)$ in the complex plane, the skew of $T$ can be approximated by $\abs{\mathtt{B}_\tr}$, in the sense that
\[ \mathrm{skew}(T)^2 \approx 1 + C(\ell) \abs{\mathtt{B}_\tr}, \]
where $\ell$ is the smallest side of $T$ and the constant $C(\ell)$ depends only on $\ell$.
\end{conjecture}

\section*{Appendix}
\subsection*{Two Technical Lemmas}

We present here the proof of two technical results, Lemmas~\ref{sign_Hess} and \ref{sign_Kxx}, used in Section~\ref{hessians_critical-points}. All computations were performed using SageMath \cite{Sage}.

\begin{lemma*}[\ref{sign_Hess}]
From the identities \eqref{HessK(x,0)} and \eqref{HessK(0,y)}, it follows that:
\begin{enumerate}
\item[(a)] If $B_1>0$, then $P_1(X_1)<0$, $P_2(X_1)>0$, $P_3(X_1)>0$. Moreover, $Q_1(Y_1)<0$, $Q_2(Y_1)<0$, $Q_3(Y_1)>0$. Consequently,
\[\Hess(K(X_1,0))>0\quad \text{and}\quad\Hess(K(0,Y_1))<0. \]
\item[(b)] If $B_1<0$, then $P_1(X_1)<0$, $P_2(X_1)<0$, $P_3(X_1)>0$. Moreover, $Q_1(Y_1)<0$, $Q_2(Y_1)>0$, $Q_3(Y_1)>0$. Consequently,
\[\Hess(K(X_1,0))<0, \qquad\Hess(K(0,Y_1))>0. \]
\end{enumerate}
\end{lemma*}

\begin{proof}
\begin{itemize}
\item $B_1>0$, $P_1(X_1) < 0 \iff \frac{125}{6}\cdot P_1(X_1) < 0 \iff$
\[ -2916 B_1^4 + 81B_1^3\sqrt{81B_1^2 + 5} + 3B_1(81B_1^2 + 5)^{3/2} - 1800B_1^2
\] \[ - 75B_1\sqrt{81B_1^2 + 5} - 100 < 0 \] 
\[ \iff \, 81B_1^3\sqrt{81B_1^2 + 5} + 3B_1(81B_1^2 + 5)^{3/2}  <  100 + 75B_1\sqrt{81B_1^2 + 5} \] \[ + 1800B_1^2 + 2916B_1^4 \] 
\[ \iff \, 8503056 B_1^8 + 1312200 B_1^6 + 66825 B_1^4 + 1125 B_1^2
  < 8503056 B_1^8 \] \[ + 10497600 B_1^6 + 437400B_1^5 \sqrt{81B_1^2 + 5} + 4278825 B_1^4 + 270000B_1^3\sqrt{81B_1^2 + 5} \] \[ + 388125 B_1^2 + 15000 B_1\sqrt{81B_1^2 + 5} + 10000 \] 
Cancelling terms with $B_1^8$ we obtain:
\[ \iff \, 1312200 B_1^6 + 66825 B_1^4 + 1125 B_1^2
  < 10497600 B_1^6 + 437400 B_1^5\sqrt{81B_1^2 + 5} \] \[ + 4278825 B_1^4 + 270000 B_1^3\sqrt{81B_1^2 + 5} + 388125 B_1^2 + 15000 B_1 \sqrt{81B_1^2 + 5} + 10000. \]

In the last inequality, the coefficients of $B_1^6$, $B_1^4$ and $B_1^2$ on the left-hand side are all less than the corresponding coefficients on the right-hand side. Hence, the above inequality is true.

\item $B_1<0$, $P_1(X_1) < 0 \iff \frac{125}{6}\cdot P_1(X_1) < 0 \iff$
\[ -2916 B_1^4 + 81 B_1^3\sqrt{81B_1^2 + 5} + 3 B_1(81B_1^2 + 5)^{3/2} - 1800 B_1^2 \] \[ - 75 B_1\sqrt{81B_1^2 + 5} - 100 < 0 \]
\[ \iff \, 75 (-B_1)\sqrt{81B_1^2 + 5} < 100 + 3 (-B_1)(81B_1^2 + 5)^{3/2} + 1800 B_1^2 + 2916 B_1^4 \] \[ + 81 (-B_1)^3\sqrt{81B_1^2 + 5}  \]
\[ \iff \, 455625 B_1^4 + 28125 B_1^2  <  17006112 B_1^8 - 472392 \sqrt{81B_1^2 + 5} B_1^7 \] \[ - 17496 B_1^5(81B_1^2 + 5)^{3/2} + 11809800 B_1^6 - 291600 B_1^5\sqrt{81B_1^2 + 5} \] \[  - 10800 B_1^3(81B_1^2 + 5)^{3/2} + 3890025 B_1^4 - 16200 B_1^3 \sqrt{81B_1^2 + 5} \] \[ - 600 B_1(81B_1^2 + 5)^{3/2} + 361125 B_1^2 + 10000. \]

When $n$ is an odd integer, the coefficients in the terms involving $B_1^n$ are negative, so the complete factor is positive. The above inequality is equivalent to:
\[ \iff \, 0 <  17006112 B_1^8 - 472392 B_1^7\sqrt{81B_1^2 + 5} - 17496 B_1^5(81B_1^2 + 5)^{3/2} \] \[ + 11809800 B_1^6 - 291600 B_1^5\sqrt{81B_1^2 + 5} - 10800 B_1^3(81B_1^2 + 5)^{3/2} + 3434400 B_1^4 \] \[ - 16200 B_1^3\sqrt{81B_1^2 + 5} - 600 B_1(81B_1^2 + 5)^{3/2} + 333000 B_1^2 + 10000, \]
By the comment above, the inequality holds.

\item If $B_1>0$, then $P_2(X_1)> 0 \iff 5\cdot P_2(X_1)>0 \iff 216 B_1^2 + 36B_1\sqrt{81B_1^2 + 5} > 0$.

\bigskip
\item If $B_1<0$, then $P_2(X_1)< 0 \iff 5\cdot P_2(X_1)<0$
\[ \iff \, 36 B_1 \sqrt{81B_1^2 + 5} < - 216 B_1^2 \]
\[ \iff \, 36(-B_1) \sqrt{81B_1^2 + 5} > 216 B_1^2 \]
\[ \iff \, 104976 B_1^4 + 6480 B_1^2 > 46656 B_1^4 \]
\[ \iff \, 58320 B_1^4 + 6480 B_1^2 > 0 \]
\[ \iff \, 9 B_1^4 + B_1^2 > 0. \]

\bigskip
\item If $B_1>0$, then $P_3(X_1)>0 \iff \frac{25}{6}\cdot P_3(X_1)>0 \iff 27 B_1^2 - 3 B_1 \sqrt{81B_1^2 + 5} + 5 > 0$
\[ \iff \, 27 B_1^2 + 5 > 3 B_1 \sqrt{81B_1^2 + 5} \]
\[ \iff \, 729 B_1^4 + 270 B_1^2 + 25 > 729 B_1^4 + 45 B_1^2 \]
\[ \iff \,  270 B_1^2 + 25 > 45 B_1^2 \]
\[ \iff \, 225 B_1^2 + 25 > 0 \]
\[ \iff \, 9 B_1^2 + 1 > 0. \]

\bigskip
\item If $B_1<0$, then $P_3(X_1)>0 \iff \frac{25}{6}\cdot P_3(X_1)>0 \iff 27 B_1^2 - 3 B_1\sqrt{81B_1^2 + 5} + 5 > 0$.
\end{itemize}

The proofs to determine the signs of $Q_k(Y_1)$ for $k\in \{1,2,3\}$ are completely analogous.
\end{proof}

\begin{lemma*}[\ref{sign_Kxx}]
\begin{enumerate}
\item[(a)] If $B_1>0$, then $K_{xx}(X_1,0)>0$.
\item[(b)] If $B_1<0$, then $K_{xx}(0,Y_1)>0$.
\end{enumerate}
\end{lemma*}

\begin{proof}
We prove the case $B_1>0$. Since
$K_{xx}(X_1,0)=-\frac{C_0}{A(X_1,0)^5}N_{xx}(X_1,0)$ and
$A(X_1,0)=(X_1^2 - B_1)^2 + \frac{1}{9} > 0$, the sign of $K_{xx}(X_1,0)$ is opposite to that of $N_{xx}(X_1,0)$. However,
\[  \frac{(125)(5)(27)}{8} N_{xx}(X_1,0)
= -7290B_1^4 - 4293\sqrt{81B_1^2 + 5}B_1^3 + 63 B_1(81B_1^2 + 5)^{3/2} \] \[ - 4500B_1^2 - 465 B_1\sqrt{81B_1^2 + 5} - 250. \]
Hence, $N_{xx}(X_1,0) < 0 \, \iff$
\[ -7290 B_1^4 - 4293 B_1^3\sqrt{81B_1^2 + 5} + 63 B_1(81B_1^2 + 5)^{3/2} - 4500 B_1^2 \] \[ - 465 B_1\sqrt{81B_1^2 + 5} - 250 < 0 \]
\[ \iff \,  63 B_1(81B_1^2 + 5)^{3/2} < 250 + 7290 B_1^4 + 4293 B_1^3\sqrt{81B_1^2 + 5} + 4500 B_1^2 \] \[ + 465 B_1\sqrt{81B_1^2 + 5} \]
\[ \iff \, 2109289329 B_1^8 + 390609135 B_1^6 + 24111675 B_1^4 + 496125 B_1^2
\] \[ < 1545961869 B_1^8 + 62591940 B_1^7\sqrt{81B_1^2 + 5} + 481150935 B_1^6
\] \[ + 45416700 B_1^5\sqrt{81B_1^2 + 5} + 61371675 B_1^4 + 6331500 B_1^3\sqrt{81B_1^2 + 5} \] \[ + 3331125 B_1^2 + 232500 B_1 \sqrt{81B_1^2 + 5} + 62500 \]

The above inequality is equivalent to:
\[ \iff \, 563327460 B_1^8 + 390609135 B_1^6 + 24111675 B_1^4 + 496125 B_1^2
\] \[ < 62591940 B_1^7\sqrt{81B_1^2 + 5} + 481150935 B_1^6 + 45416700 B_1^5 \sqrt{81B_1^2 + 5} \] \[ + 61371675 B_1^4 + 6331500 B_1^3\sqrt{81B_1^2 + 5} + 3331125 B_1^2 \] \[ + 232500 B_1 \sqrt{81B_1^2 + 5} + 62500 \]
\[ \iff \, 563327460 B_1^8 + 24111675 B_1^4 + 496125 B_1^2
< 62591940 B_1^7\sqrt{81B_1^2 + 5} \] \[ + 90541800 B_1^6 + 45416700 B_1^5 \sqrt{81B_1^2 + 5} + 61371675 B_1^4 \] \[ + 6331500 B_1^3\sqrt{81B_1^2 + 5} + 3331125 B_1^2 + 232500 B_1\sqrt{81B_1^2 + 5} + 62500 \]
\[ \iff \, 563327460 B_1^8 + 496125 B_1^2 <
62591940 B_1^7\sqrt{81B_1^2 + 5} + 90541800 B_1^6 \] \[ + 45416700 B_1^5\sqrt{81B_1^2 + 5} + 37260000 B_1^4 + 6331500 B_1^3\sqrt{81B_1^2 + 5} \] \[ + 3331125 B_1^2 + 232500 B_1\sqrt{81B_1^2 + 5} + 62500 \]
\[ \iff \, 563327460 B_1^8  <
62591940 B_1^7\sqrt{81B_1^2 + 5} + 90541800 B_1^6 \] \[ + 45416700 B_1^5\sqrt{81B_1^2 + 5} + 37260000 B_1^4 + 6331500 B_1^3\sqrt{81B_1^2 + 5} \] \[ + 2835000 B_1^2 + 232500 B_1\sqrt{81B_1^2 + 5} + 62500 \]
Taking the square,
\[ \iff \, 317337827190051600 B_1^{16}
<
317337827190051600 B_1^{16} \] \[ + 480109291311294000 B_1^{14} + 11334373826184000 B_1^{13}\sqrt{81B_1^2 + 5} \] \[ + 267902559545130000 B_1^{12} + 12888570904920000 B_1^{11}\sqrt{81B_1^2 + 5}  \] \[ + 69965137671750000 B_1^{10} + 4885879597200000 B_1^9\sqrt{81B_1^2 + 5}  \] \[ + 9880503504750000 B_1^8 + 779261998500000 B_1^7\sqrt{81B_1^2 + 5}  \] \[ + 767091161250000 B_1^6 + 58902592500000 B_1^5\sqrt{81B_1^2 + 5}  \] \[ + 31794018750000 B_1^4 + 2109712500000 B_1^3\sqrt{81B_1^2 + 5} + 624656250000 B_1^2  \] \[ + 29062500000 B_1\sqrt{81B_1^2 + 5} + 3906250000. \]

The last inequality is true since the $B_1^{16}$ terms are identical on both sides of the inequality and cancel.

To summarize,  $N_{xx}(X_1,0)<0$ and hence, $K_{xx}(X_1,0)>0$.

The proof for the case $B_1<0$ is completely analogous.
\end{proof}

\medskip

\subsection*{Acknowledgments} 
The second author is deeply grateful to the Department of Mathematics and the Faculty of Natural Sciences and Mathematics of the University of El Salvador for providing an absolutely nice ambient for being able to develop part of the present article. Special thanks to Dr. Luis Parada, former Dean of the Faculty, M.Sc. Angela Gudelia Portillo, Dean of the Faculty; Dr. Nerys Funes, Vice-Dean of the Faculty; and Dr. Dimas Tejada, Head of the Department, for their financial support and their uniformly continuous encouragement while working on the project.

The authors gratefully acknowledge Dr. Mara Denisse Rueda for extensive editing of the complete article.

The third author was supported by SECIHTI Grant 169113, \emph{Estancias Posdoctorales por México 2025 -- Modalidad Académica}, at CIMAT.

\end{document}